\documentclass[11pt]{amsart}
\usepackage{amsmath,amsthm,amsfonts,amssymb}
\usepackage[T1]{fontenc}
\usepackage{geometry}

\usepackage{hyperref}
\usepackage{enumerate}
\newcommand{\Rde}{\mathbb{R}_{\varepsilon}^n}
\newcommand{\Rdme}{\mathbb{R}_{-\varepsilon}^n}
\newcommand{\la}{\lambda}
\newcommand{\M}{\mathcal{M}}
\newcommand{\mC}{\mathcal{C}}

\newcommand{\ld}{\mathcal{L}}
\newcommand{\mL}{\mathcal{L}}

\newcommand{\pp}{p\rightarrow p}

\newcommand{\TT}{t}

\newcommand{\ki}{\|\kappa\|_{\infty}}

\newcommand{\Leb}{\Lambda}
\newcommand{\mW}{\mathcal{W}}
\newcommand{\mM}{\mathcal{M}}
\newcommand{\mD}{\mathcal{D}}
\newcommand{\ph}{\varphi}
\newcommand{\Rdp}{\mathbb{R}_+^d}
\newcommand{\Rnp}{\mathbb{R}_+^n}
\newcommand{\Rlp}{\mathbb{R}_+^l}

\newcommand{\bnH}{\tilde{\bf{H}}}
\newcommand{\pst}{\phi_p^{\ast}}

\DeclareMathOperator{\Real}{Re}

\DeclareMathOperator{\Arg}{Arg}
\DeclareMathOperator{\Dom}{Dom}
\DeclareMathOperator{\supp}{supp}

\DeclareMathOperator{\Ran}{Ran}

\newtheorem{thm}{Theorem}[section]
\newtheorem{lem}[thm]{Lemma}

\newtheorem{cor}[thm]{Corollary}
\newtheorem{pro}[thm]{Proposition}

\theoremstyle{definition}
\newtheorem{defi}[equation]{Definition}

\theoremstyle{remark}
\newtheorem*{remark}{Remark}
\newtheorem*{remark1}{Remark 1}
\newtheorem*{remark2}{Remark 2}

\numberwithin{equation}{section}

\begin{document}

\title[]{Joint spectral multipliers for mixed systems of operators}
\author[B. Wr\'obel]{B\l a\.{z}ej Wr\'obel}
\thanks{}

\address{Dipartimento di Matematica e Applicazioni, Universit\`{a} di Milano-Bicocca,
via R. Cozzi 53 I-20125, Milano, Italy, 
\newline \&
Instytut Matematyczny, Uniwersytet Wroc\l awski, pl. Grunwaldzki 2/4, 50-384 Wroc\l aw, Poland}
\email{blazej.wrobel@math.uni.wroc.pl}

\date{}
\maketitle

\begin{abstract}We obtain a general Marcinkiewicz-type multiplier theorem for mixed systems of strongly commuting operators $L=(L_1,\ldots,L_d);$ where some of the operators in $L$ have only a holomorphic functional calculus, while others have additionally a Marcinkiewicz-type functional calculus. Moreover, we prove that specific Laplace transform type multipliers of the pair $(\mathcal{L},A)$ are of certain weak type $(1,1).$ Here $\ld$ is the Ornstein-Uhlenbeck operator while $A$ is a non-negative operator having Gaussian bounds for its heat kernel. Our results include the Riesz transforms $A(\mathcal{L}+A)^{-1},$ $\mathcal{L}(\mathcal{L}+A)^{-1}.$ \end{abstract}
\numberwithin{equation}{section}
\section{Introduction}
Let $(X,\nu)$ be a $\sigma$-finite measure space. Consider a system $L=(L_1,\ldots,L_d)$ of strongly commuting non-negative self-adjoint operators on $L^2(X,\nu).$ By strong commutativity we mean that the spectral projections of $L_j,$ $j=1,\ldots,d,$ commute pairwise. In this case there exists the joint spectral resolution $E(\la)$ of the system $L.$ Moreover, for a bounded function $m\colon [0,\infty)^d\to \mathbb{C},$ the multiplier operator $m(L)$ can be defined on $L^2(X,\nu)$ by
    $$m(L)=\int_{[0,\infty)^d}m(\la)dE(\la).$$
By the (multivariate) spectral theorem, $m(L)$ is then bounded on $L^2(X,\nu).$ In this article we investigate under which assumptions on the multiplier function $m$ is it possible to extend $m(L)$ to a bounded operator on $L^p(X,\nu),$ $1<p<\infty.$

Throughout the paper we assume the $L^p(X,\nu),$ $1\leq p\leq \infty,$ contractivity of the heat semigroups corresponding to the operators $L_j,$ $j=1,\ldots,d.$ If this condition holds then we say that $L_j$ generates a symmetric contraction semigroup.

Then, by Cowling's \cite[Theorem 3]{Hanonsemi}, each of the operators $L_j,$ $j=1,\ldots,d,$ necessarily has an $H^{\infty}$ functional calculus on each $L^p(X,\nu),$ $1<p<\infty.$ This means that if $m_j$ is a bounded holomorphic function (of one complex variable) in a certain sub-sector $S_{\varphi_p}$ of the right complex half-plane, then the operator $m_j(L_j),$ given initially on $L^2(X,\nu)$ by the spectral theorem, is bounded on $L^p(X,\nu).$ However, it may happen that some of our operators also have the stronger Marcinkiewicz functional calculus. We say that $L_j$ has a Marcinkiewicz functional calculus, if every bounded function  $m_j\colon [0,\infty)\to \mathbb{C},$ which satisfies a certain Marcinkiewicz-type condition, see Definition \ref{defi:Marcon} (with $d=1$) gives rise to a bounded operator $m_j(L_j)$ on all $L^p(X,\nu),$ $1<p<\infty$ spaces. Throughout the paper we use letter $A$ to denote operators which have a Marcinkiewicz functional calculus. The formal definitions of the two kinds of functional calculi are given in Section \ref{sec:LA}.

Perhaps the most eminent difference between these functional calculi is the fact that the Marcinkiewicz functional calculus does not require the multiplier function to be holomorphic. In fact, every function which is sufficiently smooth, and compactly supported away from $0$ does satisfy the Marcinkiewicz condition.

For the single operator case various kinds of multiplier theorems have been proved in a great variety of contexts. The literature on the subject is vast; let us only name here \cite{CowDouMcYa} and \cite{Meda1} as the papers which have directly influenced our research.

As for the joint spectral multipliers for a system of commuting self-adjoint operators there are relatively fewer results. The first studied case was the one of partial derivatives $L=(\partial_1,\ldots,\partial_d),$ see \cite{Marorg} (the classical Marcinkiewicz multiplier theorem) and \cite{Horm1} (the classical H\"ormander multiplier theorem). The two theorems differ in the type of conditions imposed on the multiplier function $m$. The Marcinkiewicz multiplier theorem requires a product decay at infinity of the partial derivatives of $m,$ while the H\"ormander multiplier theorem assumes a radial decay. However, neither of the theorems is stronger than the other. Our paper pursues Marcinkiewicz-type multiplier theorems in more general contexts.

One of the first general cases of commuting operators, investigated in the context of a joint functional calculus, was that of sectorial operators (see \cite[Definition 1.1]{LanLanMer}). In \cite{Al} and \cite{AlFrMc} Albrecht, Franks, and McIntosh studied the existence of an $H^{\infty}$  joint functional calculus for a pair $L=(L_1,L_2)$ of commuting sectorial operators defined on a Banach space $B$. For some other results concerning holomorphic functional calculus for a pair of sectorial operators see \cite{LanLanMer} by Lancien, Lancien, and Le Merdy.

Marcinkiewicz-type (multivariate) multiplier theorems for specific commuting operators (i.e sublaplacians and central derivatives) on the Heisenberg (and related) groups were investigated by M\"uller, Ricci, and Stein in \cite{Mu:RiSt1}, \cite{Mu:RiSt2}, and by Fraser in \cite{Fr1}, \cite{Fr2}, \cite{Fr3}. The PhD thesis of Martini, \cite{Martini_Phd} (see also \cite{Martini_JFA} and \cite{Martini_Annales}), is a treatise of the subject of joint spectral multipliers for general Lie groups of polynomial growth. He proves various Marcinkiewicz-type and H\"ormander-type multiplier theorems, mostly with sharp smoothness thresholds.

In \cite{Sik} Sikora proved a H\"ormander-type multiplier theorem for a pair of non-negative self-adjoint operators $A_j$ acting on $L^2(X_j,\mu_j),$ $j=1,2,$ i.e.\ on separate variables\footnote{Then, the tensor products $A_1\otimes I$ and $I\otimes A_2$ commute strongly on $L^2(X_1\times X_2,\mu_1\otimes \mu_2)$}. In this article the author assumes that the kernels of the heat semigroup operators $e^{-t_jA_j},$ $t_j>0,$ $j=1,2,$ satisfy certain Gaussian bounds and that the underlying measures $\mu_j$ are doubling. Corollary \ref{corMar} of our paper is, in some sense, a fairly complete answer to a question posed in \cite[Remark 4]{Sik}.

The main purpose of the the present article is to prove (multivariate) multiplier theorems in the case when some of the considered operators have a Marcinkiewicz functional calculus, while others have only an $H^{\infty}$ functional calculus. Let us underline that, for the general results of Section \ref{sec:LA}, we only require strong commutativity and do not need that the operators in question arise from orthogonal expansions (cf. \cite{ja}) nor that they act on separate variables (cf. \cite{Sik}). In Theorem \ref{thm:LA} we show that under a certain Marcinkiewicz-type assumption on a bounded multiplier function $m$, the multiplier $m(L)$ extends to a bounded operator on $L^p(X,\nu).$ Once we realize that the only assumption we need is that of strong commutativity, the proof follows the scheme developed in \cite{ja}, \cite{erra} and \cite{jaOU}. The argument we use relies on Mellin transform techniques, together with $L^p$ bounds for the imaginary power operators, and square function estimates. For the convenience of the reader, we give a fairly detailed proof of Theorem \ref{thm:LA}.

From Theorem \ref{thm:LA} we derive two seemingly interesting corollaries. The first of these, Corollary \ref{corHinf}, gives a close to optimal $H^{\infty}$ joint functional calculus for a general system of strongly commuting operators that generate symmetric contraction semigroups. The second, Corollary \ref{corMar}, states that having a Marcinkiewicz functional calculus by each of the operators $A_j,$ $j=1,\ldots,d,$ is equivalent to having a Marcinkiewicz joint functional calculus by the system $A=(A_1,\ldots,A_d).$ Thus, in a sense, Corollary \ref{corMar} provides a most general possible Marcinkiewicz-type multiplier theorem for commuting operators.

The prototypical multipliers which fall under our theory have a product form $m_1(L_1)\cdots m_d(L_d).$ However the reader should keep in mind that Theorem \ref{thm:LA} applies to a much broader class of multiplier functions. Our condition \eqref{eq:Marcon} does not require $m$ to have a product form, but rather assumes it has a product decay. In particular Theorem \ref{thm:LA} implies $L^p,$ $1<p<\infty,$ boundedness of the imaginary power operators and Riesz transforms. In the case of a pair $(L,A)$ by imaginary powers we mean the operators $(L+A)^{iu},$ $u\in \mathbb{R},$ while by Riesz transforms we mean the operators $L(L+A)^{-1},$ $A(L+A)^{-1}$. Note however that due to the methods we use the growth of the $L^p$ norm of these operators is likely to be of order at least $(p-1)^{-4},$ $p\rightarrow 1^+.$ In particular, we do not obtain weak type $(1,1)$ results.

In Section \ref{sec:OA} we pursue a particular instance of our general setting in which  some weak type $(1,1)$ results can be proved. Namely, we restrict to the case of two operators: $\ld$ being the Ornstein-Uhlenbeck operator on $L^2(\mathbb{R}^d,\gamma)$, and $A$ being an operator acting on some other space $L^2(Y,\rho,\mu),$ where $(Y,\rho,\mu)$ is a space of homogeneous type. We also assume that the heat semigroup $e^{-tA}$ has a kernel satisfying Gaussian bounds and some Lipschitz estimates, see \eqref{sec:OA,eq:gausbound}, \eqref{sec:OA,eq:heatlipsch}, \eqref{sec:OA,eq:heatlipschngauss}. Here the operators do act on separate variables. The main result of this section is Theorem \ref{thm:OA}, which states that certain 'Laplace transform type' multipliers of the system $(\ld\otimes I,I\otimes A)$ are not only bounded on $L^p(\mathbb{R}^d\times Y, \gamma\otimes \mu),$ $1<p<\infty,$ but also from $L^1_{\gamma}(H^1(Y,\mu))$ to $L^{1,\infty}_{\gamma\otimes \mu}.$ Here $H^1(Y,\mu)$ denotes the atomic Hardy space $H^1$ in the sense of Coifman-Weiss. Section \ref{sec:OA} gives weak type $(1,1)$ results for joint multipliers in the case when one of the operators (the Ornstein-Uhlenbeck operator $\ld$, see \cite{hmm}) does not have a Marcinkiewicz functional calculus. It seems that so far such results were proved only for systems of operators all having a Marcinkiewicz functional calculus.
\section{Preliminaries}
\label{sec:Prem}

Let $L=(L_1,\ldots,L_d)$ be a system of non-negative self-adjoint operators on $L^2(X,\nu),$ for some $\sigma$-finite measure space $(X,\nu).$ We assume that the operators $L_j$ commute strongly, i.e.\ that their spectral projections $E_{L_j},$ $j=1,\ldots,d,$ commute pairwise. In this case, there exists the joint spectral measure $E$ associated with $L$ and determined uniquely by the condition
        \begin{equation*}
        L_j=\int_{[0,\infty)}\la_j dE_{L_j}(\la_j)=\int_{[0,\infty)^d} \la_j dE(\la),
        \end{equation*}
        see \cite[Theorem 4.10 and Theorems 5.21, 5.23]{schmu:dgen}.
        Consequently, for a Borel measurable function $m$ on $[0,\infty)^d,$ the multivariate spectral theorem allows us to define
         \begin{equation}
         \label{m(L)def}
        m(L)=m(L_1,\ldots,L_d)=\int_{[0,\infty)^d} m(\la) dE(\la)
        \end{equation}
        on the domain
         \begin{equation*}
    \Dom(m(L))=\bigg\{f\in L^2(X,\nu)\colon \int_{[0,\infty]^d}|m(\la)|^2dE_{f,f}(\la)<\infty\bigg\}.
    \end{equation*}
         Here $E_{f,f}$ is the complex measure defined by $E_{f,f}(\cdot)=\langle E(\cdot)f,f\rangle_{L^2(X,\nu)}.$

The crucial assumption we make is the $L^p(X,\nu)$ contractivity of the heat semigroups $\{e^{-tL_j}\},$ $j=1,\ldots,d.$ More precisely, we impose that, for each $1\leq p\leq \infty,$ and $t>0,$
\begin{equation*}
\tag{CTR}
\label{contra}
\|e^{-tL_j}f\|_{L^p(X,\nu)}\leq \|f\|_{L^p(X,\nu)},\qquad f\in L^p(X,\nu)\cap L^2(X,\nu).
\end{equation*}
This condition is often phrased as {\it the operator $L_j$ generates a symmetric contraction semigroup}. For technical reasons we often also impose
\begin{equation*}
\label{noatomatzero}
\tag{ATL}
E_{L_j}(\{0\})=0,\qquad j=1,\ldots,d.\end{equation*}
Note that under \eqref{noatomatzero} the formula \eqref{m(L)def} may be rephrased as
$$ m(L)=m(L_1,\ldots,L_d)=\int_{(0,\infty)^d} m(\la) dE(\la).$$

A particular instance of strongly commuting operators arises in product spaces, when $(X,\nu)=(\Pi_{j=1}^dX_j,\bigotimes_{j=1}^d\nu_j).$ In this case, for a self-adjoint or bounded operator $T$ on $L^2(X_j,\nu_j)$ we define
\begin{equation}
\label{eq:tensnot}
T\otimes I_{(j)}=I_{L^2(X_1,\nu_1)}\otimes \cdots \otimes I_{L^2(X_{j-1},\nu_{j-1})}\otimes T\otimes I_{L^2(X_{j+1},\nu_{j+1})} \otimes \cdots \otimes I_{L^2(X_d,\nu_d)}.
\end{equation} If $T$ is self-adjoint, then the operators $T\otimes I_{(j)}$ can be regarded as self-adjoint and strongly commuting operators on $L^2(X,\nu),$ see \cite[Theorem 7.23]{schmu:dgen} and \cite[Proposition A.2.2]{PhD}. Once again, let us point out that the general results of Section \ref{sec:LA} do not require that the operators act on separate variables. However, in Section \ref{sec:OA} we do consider a particular case of operators acting on separate variables.

Throughout the paper the following notation is used. The symbols $\mathbb{N}_0$ and $\mathbb{N}$ stand for the sets of non-negative and positive integers, respectively, while $\Rdp$ denotes $(0,\infty)^d$.

For a vector of angles $\varphi=(\varphi_1,\ldots,\varphi_d)\in (0,\pi/2]^d,$ we denote by ${\bf S}_{\varphi}$ the symmetric poly-sector (contained in the $d$-fold product of the right complex half-planes) $${\bf S}_{\varphi}=\{(z_1,\ldots,z_d)\in \mathbb{C}^d\colon z_j\neq0,\quad |\Arg(z_j)|<\varphi_j,\quad j=1,\ldots,d\}.$$ In the case when all $\varphi_j$ are equal to a real number $\varphi$ we abbreviate ${\bf S}_{\varphi}:={\bf S}_{(\varphi,\ldots,\varphi)}.$ However, it will be always clear from the context whether $\varphi$ is a vector or a number.

If $U$ is an open subset of $\mathbb{C}^d,$ the symbol $H^{\infty}(U)$ stands for the vector space of bounded functions on $U,$ which are holomorphic in $d$-variables. The space $H^{\infty}(U)$ is equipped with the supremum norm.

 If $\gamma$ and $\rho$ are real vectors (e.g.\ multi-indices), by $\gamma<\rho$ ($\gamma\leq \rho$) we mean that $\gamma_j<\rho_j$ ($\gamma_j\leq\rho_j$), for $j=1,\ldots,d.$ For any real number $x$ the symbol ${\bf x}$ denotes the vector $(x,\ldots,x)\in \mathbb{R}^d.$

        For two vectors $z,w\in \mathbb{C}^d$ we set $z^{w}=z_1^{w_1}\cdots z_d^{w_d},$ whenever it makes sense. In particular, for $\la=(\la_1,\ldots,\la_d)\in\Rdp$ and $u=(u_1,\ldots,u_d)\in\mathbb{R}^d,$ by $\la^{iu}$ we mean $\la_1^{iu_1}\cdots\la_d^{iu_d};$ similarly, for $N=(N_1,\ldots,N_d)\in\mathbb{N}^d,$ by $\la^N$ we mean $\la_1^{N_1}\cdots \la_d^{N_d}.$ This notation is also used for operators, i.e.\, for $u\in\mathbb{R}^d$ and $N\in\mathbb{N}^d$ we set
        $$L^{iu}=L_1^{iu_1}\cdots L_d^{iu_d},\qquad L^N=L_1^{N_1}\cdots L_d^{N_d}.$$
        Note that, due to the assumption on the strong commutativity, the order of the operators in the right hand sides of the above equalities is irrelevant.

        By $\langle z, w \rangle,$ $z,w\in\mathbb{C}^d$ we mean the usual inner product on $\mathbb{C}^d.$ Additionally, if instead of $w\in\mathbb{C}^d$ we take a vector of self-adjoint operators $L=(L_1,\ldots,L_d),$ then, by $\langle z, L \rangle$ we mean $\sum_{j=1}^d z_j L_j.$

        The symbol $\frac{d\la}{\la}$ (in some places we write $\frac{dt}{t}$ or $\frac{da}{a}$ instead) stands for the product Haar measure on $(\Rdp,\cdot),$ i.e.\ $$\frac{d\la}{\la}=\frac{d\la_1}{\la_1}\cdots\frac{d\la_d}{\la_d}.$$ For a function $m\in L^1(\Rdp,\, \frac{d\la}{\la}),$ we define its $d$-dimensional Mellin transform by \begin{equation}\label{eq:Mellin}\M(m)(u)=\int_{\Rdp}\la^{-iu}\,m(\la)\,\frac{d\la}{\la},\qquad u\in\mathbb{R}^d.\end{equation} It is well known that $\M$ satisfies
        the Plancherel formula
        \begin{equation*}
    \int_{\Rdp}|m(\la)|^2\,\frac{d\la}{\la}=\frac{1}{(2\pi)^{d}}\int_{\mathbb{R}^d}|\M(m)(u)|^2\,du,\qquad m\in L^2(\Rdp,\frac{d\la}{\la}),
    \end{equation*}
    and the inversion formula
    \begin{equation*}
    m(\la)=\frac{1}{(2\pi)^{d}}\int_{\mathbb{R}^d}\M(m)(u)\la^{iu}\,du,\qquad \la=(\la_1,\ldots,\la_d)\in \Rdp,
    \end{equation*}
    for $m$ such that both $m\in L^1(\Rdp,\frac{d\la}{\la})$ and $\M(m)\in L^1(\mathbb{R}^d,du).$

        Throughout the paper we use the variable constant convention, i.e.\ the constants (such as $C,$ $C_p$ or $C(p),$ etc.) may vary from one occurrence to another. In most cases we shall however keep track of the parameters on which the constant depends, (e.g.\ $C$ denotes a universal constant, while $C_p$ and $C(p)$ denote constants which may also depend on $p$). The symbol $a\lesssim b$ means that $a\leq C b,$ with a constant $C$ independent of significant quantities.

Let $B_1,B_2$ be Banach spaces and let $F$ be a dense subspace of $B_1.$ We say that a linear operator $T\colon F\to B_2$ is bounded, if it has a (unique) bounded extension to $B_1.$

\section{General multiplier theorems}
\label{sec:LA}

Throughout this section, for the sake of brevity, we write $L^p$ instead of $L^p(X,\nu)$ and $\|\cdot\|_p$ instead of $\|\cdot\|_{L^p(X,\nu)}.$ The symbol $\|\cdot\|_{p\to p}$ denotes the operator norm on $L^p.$

The first $n$ operators in the system $L_1,\ldots,L_n,$ $0\leq n\leq d$ are assumed to have an $H^{\infty}$ functional calculus. We say that a single operator $L$ has an {\it $H^{\infty}$ functional calculus on $L^p,$} $1<p<\infty$, whenever we have the following: there is a sector $S_{\varphi_p}=\{z\in \mathbb{C}\colon |\Arg (z)|<\varphi_p\},$ $\varphi_p<\pi/2,$ such that, if $m$ is a bounded holomorphic function on $S_{\varphi_p},$ then $\|m(L)\|_{L^p(X,\nu)\to L^p(X,\nu)}\leq C_{p}\|m\|_{H^{\infty}(S_{\varphi_p})}.$ The phrase '$L$ has an $H^{\infty}$ functional calculus' means that $L$ has an $H^{\infty}$ functional calculus on $L^p$ for every $1<p<\infty.$ An analogous terminology is used when considering a system of operators $L=(L_1,\ldots,L_d)$ instead of a single operator. We say that $L$ has an { \it $H^{\infty}$ joint functional calculus}, whenever the following holds: for each $1<p<\infty$ there is a poly-sector ${\bf S}_{\varphi_p},$ $\varphi_p=(\varphi_p^1,\ldots,\varphi_p^d)\in[0,\pi/2)^d,$ such that if $m$ is a bounded holomorphic function in several variables on ${\bf S}_{\varphi_p},$ then $\|m(L)\|_{L^p(X,\nu)\to L^p(X,\nu)}\leq C_{p}\|m\|_{H^{\infty}({\bf S}_{\varphi_p})}.$

The last $l$ operators in the system $L,$ i.e. $L_{n+1},\ldots,L_d,$ with $n+l=d,$ are assumed to have additionally a Marcinkiewicz functional calculus. Therefore, according with our convention, we use letter $A$ to denote these operators, i.e. $A_j=L_{n+j},$ $j=1,\ldots,l.$ In order to define the Marcinkiewicz functional calculus and formulate the main theorem of the paper we need the following definition.
\begin{defi}
\label{defi:Marcon}
We say that $m\colon\Rdp\to \mathbb{C}$ satisfies the Marcinkiewicz  condition of order $\rho=(\rho_1,\ldots,\rho_d)\in \mathbb{N}_0^d,$ if $m$ is a bounded function having partial derivatives up to order $\rho\footnote{i.e.\ $\partial^{\gamma}(m)$ exist for $\gamma=(\gamma_1,\ldots,\gamma_d)\leq \rho$},$  and for all multi-indices $\gamma=(\gamma_1,\ldots,\gamma_d)\leq\rho$
\begin{equation}\label{eq:Marcon}
\|m\|_{(\gamma)}:=\sup_{R_1,\ldots,R_d>0}\int_{R_1<\la_1<2R_1}\ldots\int_{R_d<\la_d<2R_d}|\la^{\gamma}\partial^{\gamma}m(\la)|^2\,\frac{d\la}{\la}<\infty.
\end{equation}
\end{defi}

If $m$ satisfies the Marcinkiewicz condition of order $\rho,$ then we set
\begin{equation*}
\|m\|_{Mar,\rho}:=\sup_{\gamma\leq \rho}\|m\|_{(\gamma)}.
\end{equation*}

We say that a single operator $A$ has a {\it Marcinkiewicz functional calculus}\footnote{In the single operator case it might seem better to use the term 'H\"ormander functional calculus', cf.\ \cite[Theorem 2]{Meda1}. We use the name of Marcinkiewicz to accord with the naming of the multi-dimensional condition.} of order $\rho>0$, whenever the following holds: if the multiplier function $m$ satisfies the one-dimensional (i.e.\ with $d=1$) Marcinkiewicz condition \eqref{eq:Marcon} of order $\rho,$ then the multiplier operator $m(A)$ is bounded on all $L^p(X,\nu),$ $1<p<\infty,$ and $\|m(A)\|_{L^p(X,\nu)\to L^p(X,\nu)}\leq C_{p}\|m\|_{Mar,\rho}.$  Similarly, to say that a system $A=(A_1,\ldots,A_l)$ has a {\it Marcinkiewicz joint functional calculus} of order $\rho=(\rho_1,\ldots,\rho_l)\in \Rlp$ we require the following condition to be true: if the multiplier function $m$ satisfies the $d$-dimensional Marcinkiewicz condition \eqref{eq:Marcon} of order $\rho=(\rho_1,\ldots,\rho_d),$ then the multiplier operator $m(L)$ is bounded on $L^p(X,\nu),$ $1<p<\infty,$ and
         $\|m(L)\|_{L^p(X,\nu)\to L^p(X,\nu)}\leq C_{p}\|m\|_{Mar,\rho}.$

What concerns the operators $L_1,\ldots,L_n,$ we assume that
        there exist $\theta=(\theta_1,\ldots,\theta_n)\in [0,\infty)^n$ and $\phi_p=(\phi_p^1,\ldots,\phi_p^n)\in (0,\pi/2)^n,$ such that
\begin{equation}
\label{imaL}
\|L^{iu}\|_{p\to p}\leq \mC(p,L)\,\prod_{j=1}^{n}(1+|u_j|)^{\theta_j|1/p-1/2|}\exp(\phi_p^j|u_j|),\qquad u\in \mathbb{R}^n.
\end{equation}
It can be deduced that the above condition is (essentially) equivalent to each $L_j,$ $j=1,\ldots,n,$ having an $H^{\infty}$ functional calculus on $L^p$ in the sector
$$S_{\phi_p^j}=\{z_j\in\mathbb{C}\colon |\Arg(z_j)|<\phi_p^j\},$$
see \cite[Section 5]{CowDouMcYa}. Moreover, by a recent result of Carbonaro and Dragi\v{c}evi\'c \cite{Carb-Drag} (see also \cite{Hanonsemi}), every operator for which \eqref{contra} holds satisfies \eqref{imaL} with the optimal angle $\phi_p^j=\pst:=\arcsin|2/p-1|$ and $\theta_j=\theta=3.$ Put in other words every operator generating a symmetric contraction semigroup has an $H^{\infty}$ functional calculus on $L^p$ in every sector larger than $S_{\pst}.$ The angle $\pst$ is optimal among general operators satisfying \eqref{contra}, however in many concrete cases it can be significantly sharpened.

When it comes to the operators $A_1,\ldots,A_l,$ we impose that there is a vector of positive real numbers $\sigma=(\sigma_1,\ldots,\sigma_l),$ such that for every $1<p<\infty$ and $j=1,\ldots,l$
\begin{equation}
\label{imaA}
\|A_j^{iv_j}\|_{p\to p}\leq \mC(p,A)\, \prod_{j=1}^{l}(1+|v_j|)^{\sigma_j|1/p-1/2|},\qquad v\in\Rlp.
\end{equation}
Condition \eqref{imaA} is equivalent to each $A_1,\ldots,A_l$ having a Marcinkiewicz functional calculus, see \cite[Theorem 4]{Meda1}.

For a function $m\colon S_{\phi_p}\times (0,\infty)^l\rightarrow \mathbb{C}$ and $\varepsilon\in \{-1,1\}^n$ set $$m^{\phi_p}_{\varepsilon}(\lambda,a)=m(e^{i\varepsilon_1\phi_p^1}\la_1,\ldots,e^{i\varepsilon_n\phi_p^n}\la_n,a_1,\ldots,a_l),\qquad (\la,a)\in \mathbb{R}^{n+l}_+.$$ Note that, if for fixed $a\in \mathbb{R}^l$ the function $m(\cdot,a)\in H^{\infty}(S_{\varphi_p}),$ then  the boundary value functions $\la\mapsto m^{\phi_p}_{\varepsilon}(\lambda,a)$ exist by (multivariate) Fatou's theorem. In the case when all $\phi_p^j$ are equal to one angle $\phi_p$ we abbreviate $m^{\phi_p}_{\varepsilon}=m^{(\phi_p,\ldots,\phi_p)}_{\varepsilon}.$

Throughout this section we impose the assumptions of Sections \ref{sec:Prem} and \ref{sec:LA}; in particular both \eqref{noatomatzero} and \eqref{contra} as well as \eqref{imaL} and \eqref{imaA}. The following is our main theorem.
\begin{thm}
\label{thm:LA}
Fix $1<p<\infty$ and let $m\colon {\bf S}_{\phi_p}\times \mathbb{R}^l \to \mathbb{C}$ be a bounded function with the following property: for each fixed $a\in \mathbb{R}_+^l,$ $m(\cdot,a)\in H^{\infty}({\bf S}_{\phi_p}),$ and all the functions $$\mathbb{R}^d\ni (\la,a)\mapsto m^{\phi_p}_{\varepsilon}(\la,a),\qquad \textrm{where } \varepsilon\in \{-1,1\}^{n},$$ satisfy the $d$-dimensional Marcinkiewicz condition \eqref{eq:Marcon} of some order $\rho>|1/p-1/2|(\theta,\sigma)+{\bf 1},$ where $\rho=(\rho_1,\ldots,\rho_n,\rho_{n+1},\ldots,\rho_d).$ Then the multiplier operator $m(L,A)$ is bounded on $L^p$ and
$$\|m(L,A)\|_{p\to p}\leq C_{p,d}\, \mC(p,L)\,\mC(p,A)\, \sup_{\varepsilon \in\{-1,1\}^n}\|m^{\phi_p}_{\varepsilon}(\lambda,a)\|_{Mar,\rho}.$$
\end{thm}
\begin{remark1}
 If $l=0$ ($n=d$) then we consider only operators $L_1,\ldots,L_d$ with an $H^{\infty}$ functional calculus, while if $n=0$ ($l=d$) then we consider only operators $A_1,\ldots,A_d,$ with a Marcinkiewicz functional calculus. In the latter case we do not require $m$ to be holomorphic. We only assume that it satisfies \eqref{eq:Marcon} of some order $\rho>|1/p-1/2|\sigma+{
\bf 1}.$
\end{remark1}
\begin{remark2}
 From the theorem it follows that if $m(e^{i\varepsilon\phi_p}\la,a),$ $\varepsilon\in \{-1,1\}^{n},$ satisfy the Marcinkiewicz condition of some order $\rho>\frac12(\theta,\sigma)+{\bf 1},$ then $m(L,A)$ is in fact bounded on all $L^p$ spaces, $1<p<\infty.$
\end{remark2}

Before proving Theorem \ref{thm:LA} let us first state and prove two corollaries.

The first of these corollaries provides an $H^{\infty}$ joint functional calculus for a general system of strongly commuting operators $L_j,$ $j=1,\ldots,d,$ satisfying \eqref{contra} and \eqref{noatomatzero}. Corollary \ref{corHinf} generalizes \cite[Theorem 1]{Carb-Drag} to systems of commuting operators; although it is slightly weaker than \cite[Theorem 1]{Carb-Drag} in the case $d=1.$ Recall that $\pst=\arcsin|2/p-1|.$
\begin{cor}
\label{corHinf}
Let $L=(L_1,\ldots,L_d)$ be a general system of non-negative self-adjoint strongly commuting operators that satisfy both \eqref{contra} and \eqref{noatomatzero}. Fix $1<p<\infty$ and let $m$ be a bounded holomorphic functions of $d$-variables in ${\bf S}_{\pst}.$ If for some $\rho>(5/2,\ldots,5/2)$ we have
$$\sup_{\varepsilon \in\{-1,1\}^d}\|m^{\pst}_{\varepsilon}(\lambda,a))\|_{Mar,\rho}<\infty,$$
then $m(L)$ is bounded on $L^p$ and
 $$\|m(L)\|_{p\to p}\leq C_{p,d}\, \mC(p,L)\, \sup_{\varepsilon \in\{-1,1\}^d}\|m^{\pst}_{\varepsilon}(\lambda,a))\|_{Mar,\rho}.$$
\end{cor}
\begin{proof}
Using \cite[Theorem 1]{Carb-Drag} to the imaginary powers $L_j^{iu_j},$ $j=1,\ldots,d,$ and interpolating with the bound $\|L_j^{iu_j}\|_{2\to 2}\leq 1,$ we obtain \eqref{imaL} with arbitrary $\theta_j/2>3/2$ and $\phi_p^j=\pst.$ Now, an application of Theorem \ref{thm:LA} (with $n=d$) gives the desired boundedness.
\end{proof}
\begin{remark1}
Note that, as we do not require $m$ to be holomorphic in a bigger sector, our theorem is stronger than a combination of \cite[Theorem 5.4]{AlFrMc} and \cite[Theorem 1]{Carb-Drag} given in \cite[Proposition 3.2]{jaRieszGen}.
\end{remark1}
\begin{remark2}
 Examples of multiplier functions satisfying the assumptions of the corollary include $m_{j}^{\sigma}(\la)=\la_j^{\sigma}/(\la_1+\ldots \la_d)^{-\sigma},$ where $\sigma>0.$ The operators $m_{j}^{\sigma}(L),$ $j=1,\ldots,d,$ are intimately connected with the Riesz transforms, see \cite{jaRieszGen}.
\end{remark2}

The second corollary treats the case when all the considered operators have a Marcinkiewicz functional calculus, i.e.\ $n=0$ and $l=d.$ It implies that a system $A=(A_1,\ldots,A_d)$ has a Marcinkiewicz joint functional calculus of a finite order if and only if each $A_j,$ $j=1,\ldots,d,$ has a Marcinkiewicz functional calculus of a finite order.
\begin{cor}
\label{corMar}
We have the following:
\begin{itemize}
\item[(i)] If, for each $j=1,\ldots,d,$ the operator $A_j$ has a Marcinkiewicz functional calculus of order $\rho_j,$ then the system $A=(A_1,\ldots,A_d)$ has a Marcinkiewicz joint functional calculus of every order greater than $\rho+{\bf 1}.$
\item[(ii)] If the system $A=(A_1,\ldots,A_d)$ has a Marcinkiewicz joint functional calculus of order $\rho,$ then, for each $j=1,\ldots,l,$ the operator $A_j$ has a Marcinkiewicz functional calculus of order $\rho_j.$
\end{itemize}
\end{cor}
\begin{proof}
To prove item (i), note that having a Marcinkiewicz functional calculus of order $\rho_j$ implies satisfying \eqref{imaA} with every $\sigma_j>2\rho_j$. This observation follows from the bounds $\|A_j^{iv_j}\|_{p\to p}\leq C_p (1+|v_j|)^{\rho_j},$ $1<p<\infty,$ and $\|A_j^{iv_j}\|_{2\to 2}\leq 1,$ together with an interpolation argument. Now, Theorem \ref{thm:LA} (with $n=0$ and $l=d$) implies the desired conclusion.

The proof of item (ii) is even more straightforward, we just need to consider functions $m_j,$ $j=1,\ldots,d,$ which depend only on the variable $\la_j.$
\end{proof}
\begin{remark}
 The most typical instance of strongly commuting operators arises on product spaces, when each $A_j$ initially acts on some $L^2(X_j,\nu_j).$ Moreover, there are many results in the literature, see e.g.\ \cite{s1,s2,s4,s6,s7,vectvalMM,Thanherm}, which imply that a single operator has a Marcinkiewicz functional calculus. Consequently, using the corollary we obtain a joint Marcinkiewicz functional calculus for a vast class of systems of operators acting on separate variables. In particular, we may take $m(\la)=1-(\la_1+\cdots + \la_d)^{\delta}\chi_{\la_1+\cdots \la_d\leq 1},$ for $\delta>0$ large enough, thus obtaining the boundedness of the Bochner-Riesz means for the operator $A_1+\cdots + A_d.\footnote{More formally, we mean here $A_1\otimes I_{(1)}+\cdots + A_d\otimes I_{(d)},$ with the summands given by \eqref{eq:tensnot}}$ However, because of the assumed generality, these results are by no means optimal.
\end{remark}

To prove Theorem \ref{thm:LA} we need two auxiliary results which seem interesting on their own. First we need the $L^p$ boundedness of the square function
\begin{equation}
\label{squaref}
g_N(f)^2=\int_{(0,\infty)^{d}}\left|t^NL^Ne^{-\langle t ,L\rangle }f\right|^2\,\frac{dt}{t};
\end{equation}
recall that
\begin{align*}
t^NL^N&=(t_1L_1)^{N_1}\cdots (t_d L_d)^{N_d},\\
\langle t ,L\rangle&=t_1L_1+\cdots+t_dL_d.
\end{align*}
This will be proved as a consequence of a $d$-dimensional variant of \cite[Theorem 5.3]{AlFrMc} due to Albrecht, Franks and McIntosh.
\begin{thm}[cf.\ {\cite[Theorem 2.4]{ja}}]
\label{thm:gfun}
For each fixed $N\in\mathbb{N}^d$ the square function $g_N$ given by \eqref{squaref} preserves the $L^p$ norm, i.e.
\begin{equation*}
c_{p,N}\|f\|_p\leq \|g_N(f)\|_p\leq C_{p,N} \|f\|_p,\qquad 1<p<\infty.
\end{equation*}
\end{thm}
\begin{proof}[Proof (sketch)]
Even though \cite[Theorem 5.3]{AlFrMc} is given only for $d=2$ it readily generalizes to systems of $d$ operators, with the same assumptions. Hence, we just need to check that these assumptions are satisfied.

Setting $h_j(z)=z^{N_j}e^{-z},$ $z\in\mathbb{C},$ we clearly see that $h_j\in H^{\infty}(S_{\mu})$ for every $\mu < \pi/2,$ and
$$|h_j(z)|\leq C_\mu \frac{|z|}{1+|z|^2},\qquad z\in S_{\mu} $$
In the terminology of \cite{AlFrMc} this means that $h_j \in \Psi(S_{\mu}),$ for every $\mu<\pi/2.$
Observe also that our square function is of the form
$$g_N(f)^2=\int_{(0,\infty)^{d}}\left|h_1(t_1L_1)\cdots h_d(t_d L_d)f\right|^2\,\frac{dt}{t}.$$

Fix $j=1,\ldots,d,$ and denote $T=L_j.$ By referring to the $d$-dimensional version of \cite[Theorem 5.3]{AlFrMc} we are left with verifying that: $T$ is of a type $\omega<\pi/2$ (see \cite[p.\ 293]{AlFrMc} for a definition), $T$ is one-one, and both $\Dom T$ and $\Ran T$ are dense in the Banach space $B:=L^p(X,\nu).$ The reader is kindly referred to consult the proof of \cite[Proposition 3.2]{jaRieszGen}, where a justification of these statements is contained

A more detailed and slightly different proof of the proposition can be given along the lines of the proof of \cite[Corollary 4.1.2]{PhD}.
\end{proof}

For fixed $N\in\mathbb{N}^d$ and a parameter $t=(t_1,\ldots,t_d)\in(0,\infty)^{d}$ we set
\begin{equation*}
m_{N,t}(\la)=\prod_{j=1}^d(t_j\lambda_j)^{N_j} \exp\bigg(-\sum_{j=1}^d t_j \la_j\bigg)\,m(\la).
\end{equation*}
Recall that the Mellin transform $\M$ is given by \eqref{eq:Mellin}, while $L^{iu}=L^{iu_1}\cdots L^{iu_d},$ with $L_{n+j}=A_j$ and $u_{n+j}=v_j,$ for $j=1,\ldots,l.$ Theorem \ref{thm:LA} will be deduced from the following.
\begin{thm}[Cf.\ {\cite[Theorem 1]{Meda1} and \cite[Theorem 2.2]{ja}}]
\label{thm:gen}
Let $L=(L_1,\ldots,L_d),$ be a general system of non-negative self-adjoint operators satisfying \eqref{contra} and \eqref{noatomatzero} and let $1<p<\infty$ be fixed. If $m\colon (0,\infty)^{d}\to \mathbb{C}$ is a bounded function such that, for some $N\in\mathbb{N}^d,$
\begin{equation*}
m(L,N,p):=\int_{\mathbb{R}^{d}}\sup_{t\in(0,\infty)^{d}}|\M(m_{N,t})(u)|\,\|L^{iu}\|_{\pp}\,du<\infty,
\end{equation*}
then the multiplier operator $m(L)$ is bounded on $L^p(X,\nu)$ and
$$\|m(L)\|_{p\to p}\leq C_{p,d,N}m(L,N,p).$$
\end{thm}
\begin{proof}
The proof follows the scheme developed in the proof of \cite[Theorem 1]{Meda1} and continued in the proof of \cite[Theorem 2.2]{ja}, however, for the convenience of the reader we provide details.

All the needed quantities are defined on $L^2 \cap L^p$ by the multivariate spectral theorem. From the inversion formula for the Mellin transform and the multivariate spectral theorem we see that
\begin{equation}
\label{eq:auxeq0}
t^{N}L^N\exp(-2^{-1}\langle t, L\rangle)m(L)f=\frac{1}{(2\pi)^d}\int_{\mathbb{R}^d}\M(m_{N,t})(u)L^{iu}f\,du.
\end{equation}

Consequently, since $t^{1} L^{1}\exp(-2^{-1}\langle t, L\rangle)$ is bounded on $L^2,$ we have
\begin{equation}
\label{eq:auxeq}
t^{N+{ 1}}L^{N+{ 1}}\exp(-\langle t, L\rangle)m(L)f=\frac{1}{(2\pi)^d}\int_{\mathbb{R}^d}\M(m_{N,t})(u)t^{ 1}L^{ 1}\exp\big(-\frac12\langle t, L\rangle\big)(L^{iu}f)\,du.
\end{equation}
Note that, for each fixed $t\in\Rdp,$ both the integrals in \eqref{eq:auxeq0} and \eqref{eq:auxeq} can be considered as Bochner integrals of (continuous) functions taking values in $L^2.$

Then, at least formally, from Theorem \ref{thm:gfun} followed by \eqref{eq:auxeq}, we obtain
\begin{align*}
&(C_{p,d,N+{ 1}})^{-1}\|m(L)f\|_p\leq \|g_{N+{ 1}}(m(L)(f))\|_p\\
&=\bigg\|\bigg(\int_{\Rdp}\bigg|\frac{1}{(2\pi)^d}\int_{\mathbb{R}^d}\M(m_{N,t})(u)tL\exp(-2^{-1}\langle t, L\rangle)(L^{iu}f)\,du\bigg|^2\,\frac{dt}{t}\bigg)^{1/2}\bigg\|_p.
\end{align*}
Hence, using Minkowski's integral inequality, it follows that $\|m(L)f\|_p$ is bounded by
\begin{equation*} (2\pi)^{-d}\,C_{p,d,N+{ 1}}\int_{\mathbb{R}^d}\sup_{\TT\in\Rdp}|\M(m_{N,\TT})(u)| \bigg\|\bigg(\int_{\Rdp}\bigg|tL\exp(-2^{-1}\langle t, L\rangle)(L^{iu}f)\bigg|^2\,\frac{dt}{t}\bigg)^{1/2}\bigg\|_p\,du.
\end{equation*}
Now, observing that
$$\left(\int_{\Rdp}\bigg|tL\exp(-2^{-1}\langle t, L\rangle)(L^{iu}f)\bigg|^2\,\frac{dt}{t}\right)^{1/2}=2^{d}g_{{ 1}}(L^{iu}f)$$
and using once again Theorem \ref{thm:gfun} (this time with $N={1}$), we arrive at
\begin{align*}
\|m(L)f\|_p&\leq \pi^{-d}\,C_{p,d,N+{1}}\int_{\mathbb{R}^d}\|g_{{ 1}}(L^{iu}f)\|_p\,\sup_{\TT\in \Rdp}|\M(m_{N,\TT})(u)|\,du \\ 
&\leq \pi^{-d}\,C_{p,d,N+{ 1}}C_{p,d,{ 1}}\int_{\mathbb{R}^d}\|L^{iu}\|_{p\to p}\,\sup_{\TT\in\Rdp}|\M(m_{N,\TT})(u)|\,du\,\|f\|_p.
\end{align*}
Thus, the proof of Theorem \ref{thm:gen} is finished, provided we justify the formal steps above. This however can be done almost exactly as in \cite[p. 642]{Meda1}. We omit the details here and kindly refer the interested reader to \cite[p.\ 24]{PhD}.
\end{proof}
\begin{remark}
The proof of Theorem \ref{thm:gen} we present here is modeled over the original proof of \cite[Theorem 1]{Meda1} for the one-operator case. In \cite[Theorem 2.1]{Hanonultracon} the authors gave a simpler proof of \cite[Theorem 1]{Meda1}. However, a closer look at their method reveals that it does not carry over to our multivariate setting. The reason is that we initially do not know whether multivariate multipliers of Laplace transform type $\Rdp\ni\la\mapsto \la_1\cdots \la_d \int_{\Rdp}\exp(-t_1\la_1+\cdots t_d \la_d)\, \kappa(t)\,dt,$
with $\kappa$ being a bounded function on $\Rdp$ that may not have a product form,  produce bounded multiplier operators on $L^p.$
\end{remark}
Having proved Theorem \ref{thm:gen} we proceed to the proof of our main result.
\begin{proof}[Proof of Theorem \ref{thm:LA}]
The proof is based on applying Theorem \ref{thm:gen} to the system $(L_1,\ldots,L_d)$ with $L_{n+j}=A_j,$ $j=1,\ldots,l.$ Note that here the distinction between the operators $L_j,$ $j=1,\ldots,n,$ and $A_j,$ $j=1,\ldots,l,$ is relevant. The assumptions \eqref{imaL} and \eqref{imaA} imply that it is enough to verify the bound
\begin{equation}
\label{eq:decayIm}
\begin{split}
&\sup_{t\in(0,\infty)^{n+l}}|\M(m_{N,t})(u,v)|\\ 
&\lesssim \prod_{j=1}^n (1+|u_j|)^{-\rho_j}\exp(-\phi_p^j|u_j|)\prod_{j=1}^l (1+|v_{j}|)^{-\rho_{n+j}}\ \max_{\varepsilon \in\{-1,1\}^d}\|m(e^{i\varepsilon\phi_p}\cdot,\cdot)\|_{Mar,\rho},
\end{split}
\end{equation}
uniformly in $(u,v)\in \mathbb{R}^n\times \mathbb{R}^l.$ The Mellin transform in \eqref{eq:decayIm} is
$$\M(m)(u,v)=\int_{\Rnp}\int_{\Rlp}m(\la,a)\la^{-iu}a^{-iv}\frac{d\la}{\la}\frac{da}{a},$$
where $a^{-iv}=a_1^{-iv_1}\cdots a_l^{-iv_l}$ while $\frac{da}{a}=\frac{da_1}{a_1}\cdots \frac{da_l}{a_l}.$
Throughout the proof we will sometimes use $\la=(\la_1,\ldots,\la_n,\la_{n+1},\ldots,\la_d)$ and $u=(u_1,\ldots,u_n,u_{n+1},\ldots,u_d)$ to denote the variables $(\la_1,\ldots,\la_n,a)$ and $(u_1,\ldots,u_n,v).$ In such instances we understand that $\la_{n+j}=a_j$ and $u_{n+j}=v_j,$ $j=1,\ldots,l.$

The proof of \eqref{eq:decayIm} is an appropriately adjusted combination of the proofs of \cite[Theorem 4.2]{jaOU} and \cite[Theorem 4.1]{ja}, based on the usage of Theorem \ref{thm:gen}. The main idea is to change the path of integration in the first $n$ variables under the integral in \eqref{eq:decayIm}. This approach originates in \cite[Theorem 2.2]{funccalOu}. The proof we present here is a multivariate generalization of both the proofs of \cite[Theorem 2.2]{funccalOu} and \cite[Theorem 4]{Meda1}. For the sake of completeness we give details.

Defining $\Rde=\{x\in\mathbb{R}^d\colon \varepsilon_j x_j\geq 0,\, j=1,\ldots,d\},$ with $\varepsilon\in \{-1,1\}^n,$ we note that it suffices to obtain \eqref{eq:decayIm} separately on each $\Rde\times \mathbb{R}^l.$ Thus, till the end of the proof we fix $\varepsilon\in \{-1,1\}^n$ and take $u \in \Rdme.$ By our assumptions, for each fixed $a\in \Rlp,$ $N\in \mathbb{N}^d,$ $t\in \Rnp$ and $u\in\mathbb{R}^n,$  the function \begin{align*} m_{N,t}(z,a)z^{-iu-{\bf 1}}=&t^{N}z^{(N_1,\ldots,N_n)-iu-{\bf 1}}a^{(N_{n+1},\ldots,N_d)-iv-{\bf 1}}m(z,a)\\
& \times \exp(-2^{-1}\langle z, (t_1,\ldots,t_n)\rangle-2^{-1}\langle a, (t_{n+1},\ldots,t_d)\rangle)\end{align*} is bounded and holomorphic on $${\bf S}_{\phi_p}=\{z\in \mathbb{C}^n\colon |\Arg(z_j)|\leq \phi_p^j,\qquad j=1,\ldots,n\}.$$ Moreover, $m_{N,t}(z,a)z^{-iu-{1}}$ is rapidly (exponentially) decreasing when $\Real(z_j)\to \infty,$ $j=1,\ldots,n.$ Thus, for each $\varepsilon\in\{-1,1\}^n,$ we can use (multivariate) Cauchy's integral formula to change the path of integration in the first $n$ variables of the integral defining $\M(m_{N,t})(u,v)$ to the poly-ray $\{(e^{i\varepsilon_1\phi_p^1}\la_1,\ldots,e^{i\varepsilon_n\phi_p^n}\la_n)\colon \la \in \Rnp\}.$ Then, denoting
$\tilde{m}:=m_{\varepsilon}^{\phi_p}$ and $\varepsilon\phi_p=(\varepsilon_1\phi_p^1,\ldots,\varepsilon_n\phi_p^n),$
 we  obtain
\begin{equation}
\label{eq:Mform}
\begin{split}
&e^{-\langle u,(\varepsilon \phi_p)\rangle}e^{-i\langle N,(\varepsilon \phi_p)\rangle}\M(m_{N,t})(u,v)\\
&=\int_{\Rnp}\int_{\Rlp}t^N(\la,a)^N\exp\big(-\frac12\langle (e^{i\varepsilon_1\phi_p^1}t_1,\ldots, e^{i\varepsilon_d\phi_p^n}t_n,t_{n+1}\ldots,t_d),(\la,a)\rangle\big)\\
&\times \tilde{m}(\la,a)\la^{-iu}a^{-iv}\,\frac{d\la}{\la}\frac{da}{a}:=\int_{\Rdp}\tilde{m}_{N,t}(\la)\la^{-iu}\,\frac{d\la}{\la}=\M(\tilde{m}_{N,t})(u).
\end{split}
\end{equation}
In te second to the last equality above it is understood that $u\in \mathbb{R}^d$ and $\la \in \Rdp$ with $\la_{n+j}=a_j,$  $u_{n+j}=v_j,$ for $j=1,\ldots,l;$ while $\frac{d\la}{\la}$ denotes the Haar measure on $(\Rdp,\cdot).$

We claim that, for $u\in \mathbb{R}^d,$
\begin{equation}
\label{sec:LA,eq:claimGen}
\sup_{t\in\Rdp}\left|\int_{\Rdp}\tilde{m}_{N,t}(\la)\la^{-iu}\,\frac{d\la}{\la}\right|\leq C_{N,\rho} \prod_{j=1}^n (1+|u_j|)^{-\rho_j|1/p-1/2|} \max_{\varepsilon \in \{-1,1\}^n}\|m(e^{i\varepsilon\phi_p}\cdot,\cdot)\|_{Mar,\rho}.
\end{equation}
Once the claim is proved, coming back to \eqref{eq:Mform} we obtain \eqref{eq:decayIm} for $u\in\Rdme$ and $v\in\mathbb{R}^l,$ hence, finishing the proof of Theorem \ref{thm:LA}.

Thus, till the end of the proof we focus on justifying \eqref{sec:LA,eq:claimGen}. Let $N\in \mathbb{N}^d,$ $N>\rho,$ and $\psi$ be a nonnegative, $C^{\infty}$ function supported in $[1/2,2]$ and such that $$\sum_{k=-\infty}^{\infty}\psi(2^k v )=1,\qquad v>0.$$ Then, for $\Psi_{k}(\la)=\psi(2^{k_1}\la_1)\cdots \psi(2^{k_d}\la_d),$ $$\sum_{k\in\mathbb{Z}^d}\Psi_{k}(\la)=1, \qquad \la\in\mathbb{R}^d_+.$$

Set $$c_{N_j,\rho_j,u_j}=\frac{(-1)^{\rho_j}}{(N_j-iu_j)\cdots (N_j-iu_j+\rho_j-1)}\qquad \textrm{and}\qquad c_{N,\rho,u}=\prod_{j=1}^d c_{N_j,\rho_j,u_j}.$$ Changing variables $t_j \la_j\to \la_j$ and integrating by parts $\rho_j$ times in the $j$-th variable, $j=1,\ldots,d,$ we see that
\begin{align*}
\M(\tilde{m}_{N,t})(u)=c_{N,\rho,u}t^{iu}\sum_{k\in \mathbb{Z}^d}\int_{\Rdp}\la^{N+\rho-iu}\partial^{\rho}\bigg(e^{-2^{-1}\langle w,\la\rangle }\tilde{m}(\la_1/t_1,\ldots,\la_d/t_d)\Psi_k(\la)\bigg)\,\frac{d\la}{\la},
\end{align*}
where $w\in \mathbb{C}^n\times \Rlp$ is the vector $w=(e^{i\varepsilon_1\phi_p^1},\cdots,e^{i\varepsilon_n\phi_p^n},1,\ldots,1).$
For further reference note that $\Real(w_j)>0,$ for each $j=1,\ldots,d.$

Leibniz's rule allows us to express the derivative $\partial^{\rho}$ as a weighted sum of derivatives of the form
\begin{align*}E^{k}_{\gamma,\delta,t}(\la)&= e^{-2^{-1}\langle w,\la\rangle}t^{-\gamma}(\partial^{\gamma}\tilde{m})(\la_1/t_1,\ldots,\la_d/t_d)2^{\langle k,\delta\rangle}
\prod_{j=1}^d\bigg(\frac{d^{\delta_j}}{d\la_j^{\delta_j}}\psi\bigg)(2^{k_j}\la_j),\end{align*}
where $\gamma=(\gamma_1,\ldots,\gamma_d)$ and $\delta=(\delta_1,\ldots,\delta_d)$ are multi-indices such that $\gamma+\delta\leq\rho.$ Proceeding further as in the proof of \cite[Theorem 4]{Meda1}, we denote
$$I_{k,N,\gamma,\delta}(t,u)\equiv\int_{\Rdp}\la^{N+\rho-iu}E^{k}_{\gamma,\delta,t}(\la)\,\frac{d\la}{\la}.$$
Set $p_k=p_{k_1}\cdots p_{k_d}$ with $p_{k_j},$ $j=1,\ldots,d,$ given by
$$p_{k_j}=\begin{cases}
 2^{-k_j\rho_j}, &\mbox{if $k_j>0$}, \\
  2^{-k_j(N_j+\rho_j)}\exp(-2^{-k_j-2}\Real{w_j}), &\mbox{if $k_j\leq0,$}
       \end{cases} $$
sot that $\sum_{k\in\mathbb{Z}^d}p_k<\infty.$

Observe that it is enough to verify the bound
\begin{equation}
\label{sec:LA,eq:post1}
|I_{k,N,\gamma,\delta}(t,u)|\leq C_{N,\gamma,\delta}\,\|\tilde{m}\|_{Mar,\rho}\, p_{k}, \qquad k\in \mathbb{Z}^d,
\end{equation}
uniformly in $t\in\Rdp$ and $u\in\mathbb{R}^d.$ Indeed, assuming \eqref{sec:LA,eq:post1} we obtain
\begin{align*}
\nonumber
\sup_{t\in\Rdp}\left|\int_{\Rdp}\tilde{m}_{N,t}(\la)\la^{-iu}\,\frac{d\la}{\la}\right|&\leq C_N\prod_{j=1}^{d}(1+|u_j|)^{-\rho_j}\sum_{\gamma+\delta\leq \rho}C_{\gamma,\delta,\rho}\,\sum_{k\in \mathbb{Z}^d}\sup_{t\in\Rdp}|I_{k,N,\gamma,\delta}(t,u)|\\
&\leq C_{N,\rho}\|\tilde{m}\|_{Mar,\rho} \prod_{j=1}^{d}(1+|u_j|)^{-\rho_j},
\end{align*}
and \eqref{sec:LA,eq:claimGen} follows. Thus, it remains to show \eqref{sec:LA,eq:post1}.

From the change of variable $2^{k_j}\la_j\to \la_j$ we have
\begin{align*}
&|I_{k,N,\gamma,\delta}(t)|=2^{-\langle k,N+\rho-\gamma-\delta\rangle }\left|\int_{[1/2,2]^{d}}\la^{N+\rho-\gamma-iu}\exp(-2^{-1}\langle {\bf 2}^{-k}w,\la\rangle)\right. \\
&\times\left.\left(\frac{\la_1}{2^{k_1}t_1},\ldots,\frac{\la_d}{2^{k_d}t_d}\right)^{\gamma}\partial^{\gamma}(\tilde{m})\left(\frac{\la_1}{2^{k_1}t_1},\ldots,\frac{\la_d}{2^{k_d}t_d}\right)\partial^{\delta}(\Psi)(\la)\,\frac{d\la}{\la}\right|.
\end{align*}
Thus, applying Schwarz's inequality we obtain
\begin{equation}
\label{sec:LA,eq:estI}
\begin{split}
&|I_{k,N,\gamma,\delta}(t)|\leq C_{\Psi} 2^{-\langle k,N+\rho-\gamma-\delta\rangle} \left(\int_{[1/2,2]^d}\left|\la^{N+\rho-\gamma}\exp(-2^{-1}\langle {\bf 2}^{-k}\Real(w),\la\rangle)\right|^2\,\frac{d\la}{\la}\right)^{1/2}\\
&\times \left(\int_{[1/2,2]^d}\left|\left(\frac{\la_1}{2^{k_1}t_1},\ldots,\frac{\la_d}{2^{k_d}t_d}\right)^{\gamma}\partial^{\gamma}(m)\left(\frac{\la_1}{2^{k_1}t_1},\ldots,\frac{\la_d}{2^{k_d}t_d}\right)\right|^2\,\frac{d\la}{\la} \right)^{1/2}.
\end{split}
\end{equation}
Moreover, since $\Real(w_j)>0,$ for $j=1,\ldots,d,$ it is not hard to see that
\begin{equation}
\label{sec:LA,eq:estint}
\begin{split}
&\left(\int_{[1/2,2]^d}\left|\la^{N+\rho-\gamma}\exp(-2^{-1}\langle{\bf 2}^{-k}\Real(w),\la\rangle)\right|^2\,\frac{d\la}{\la}\right)^{1/2}\\ 
&=\left(\prod_{j=1}^{d}\int_{[1/2,2]}\left|\la_j^{N_j+\rho_j-\gamma_j}\exp(-2^{-k_j-1}\Real(w_j)\la_j)\right|^2\,\frac{d\la_j}{\la_j}\right)^{1/2}\\
&\leq C_{N,\rho,\gamma}\prod_{j=1}^{d} \begin{cases}
 1, &\mbox{if $k_j>0$}, \\
  \exp(-2^{-k_j-2}\Real(w_j)), &\mbox{if $k_j\leq0.$}
       \end{cases}
\end{split}
\end{equation}

Now, coming back to \eqref{sec:LA,eq:estI}, we use the assumption that $\tilde{m}$ satisfies the Marcinkiewicz condition of order $\rho$ together with \eqref{sec:LA,eq:estint} (recall that $\gamma+\delta\leq \rho< N$) to obtain \eqref{sec:LA,eq:post1}. The proof of Theorem \ref{thm:LA} is thus finished.
\end{proof}

\section{Weak type results for the system $(\ld,A)$}
 \label{sec:OA}
Here we consider the pair of operators $(\mL\otimes I,I\otimes A),$ where $\mL$ is the $d$-dimensional Ornstein-Uhlenbeck (OU) operator, while $A$ is an operator having certain Gaussian bounds on its heat kernel (which implies that $A$ has a Marcinkiewicz functional calculus). We also assume that $A$ acts on a space of homogeneous type $(Y,\zeta,\mu).$ The main theorem of this section is Theorem \ref{thm:OA}. It states that Laplace transform type multipliers of $(\mL\otimes I,I\otimes A)$ are bounded from the $H^1(Y,\mu)$-valued $L^1(\mathbb{R}^d,\gamma)$ to $L^{1,\infty}(\gamma \otimes \mu).$ Here $H^1(Y,\mu)$ is the atomic Hardy space in the sense of Coifman and Weiss \cite{CW}, while $\gamma$ is the Gaussian measure on $\mathbb{R}^d$ given by $d\gamma(x)=\pi^{-d/2}e^{-|x|^2}dx.$ Additionally, in the appendix we show that the considered weak type $(1,1)$ property interpolates well with the boundedness on $L^2,$ see Theorem \ref{thm:interH1}.

In what follows we denote by $\mL$ the $d$-dimensional Ornstein-Uhlenbeck operator
$$
-\frac{1}{2}\Delta+\langle x,\nabla\rangle.
$$
It is easily verifiable that $\mL$ is symmetric on $C_c^{\infty}(\mathbb{R}^d)$ with respect to the inner product on $L^2(\mathbb{R}^d,\gamma).$ The operator $\mL$ is also essentially self-adjoint on $C_c^{\infty}(\mathbb{R}^d),$ and we continue writing $\mL$ for its unique self-adjoint extension.

It is well known that $\mL$ can be expressed in terms of Hermite polynomials by
$$\mL f=\sum_{k\in\mathbb{N}^d_0} |k| \langle f,\bnH_k\rangle_{L^2(\mathbb{R}^d,\gamma)}\bnH_k=\sum_{j=0}^{\infty}jP_jf,$$
on the natural domain
$$\Dom(\mL)=\{f\in L^2(\mathbb{R}^d,\gamma)\colon \sum_{k\in\mathbb{N}^d_0}|k|^2 \langle f,\bnH_k\rangle_{L^2(\mathbb{R}^d,\gamma)}<\infty\}.$$
Here $|k|=k_1+\cdots+k_d$ is the length of a multi-index $k\in \mathbb{N}^d_0,$ $\bnH_k$ denotes the $L^2(\mathbb{R}^d,\gamma)$ normalized $d$-dimensional Hermite polynomial of order $k,$ while
$$P_jf=\sum_{|k|=j}\langle f,\bnH_k\rangle_{L^2(\mathbb{R}^d,\gamma)} \bnH_k,\qquad j\in\mathbb{N}_0,$$ is the projection onto the eigenspace of $\mL$ with eigenvalue $j.$

For a bounded function $m\colon\mathbb{N}_0\to \mathbb{C},$ the spectral multipliers $m(\mL)$ are defined by \eqref{m(L)def} with $d=1$. In the case of the Ornstein-Uhlenbeck operator they are given by
\begin{equation*}
m(\mL)f=\sum_{k\in\mathbb{N}^d_0} m(k_1+\cdots + k_d) \langle f,\bnH_k\rangle_{L^2(\mathbb{R}^d,\gamma)}\bnH_k=\sum_{j=0}^{\infty}m(j)P_jf.
\end{equation*}

Let $m$ be a function, which is bounded on $[0,\infty)$ and continuous on $\mathbb{R}_+.$ We say that $m$ is an $L^p(\mathbb{R}^d,\gamma)$-uniform multiplier of $\mL,$ whenever
$$\sup_{t>0}\|m(t \mL)\|_{L^p(\mathbb{R}^d,\gamma)\to L^p(\mathbb{R}^d,\gamma)}<\infty.$$
Observe that by the spectral theorem the above bound clearly holds for $p=2.$ Using \cite[Theorem 3.5 (i)]{hmm} it follows that, if $m$ is an $L^p(\mathbb{R}^d,\gamma)$-uniform multiplier of $\mL$ for some $1<p<\infty,$ $p\neq 2,$ then $m$ necessarily extends to a holomorphic function in the sector $S_{\phi_p^{*}}$ (recall that $\pst=\arcsin|2/p-1|$). Assume now that $m(t\mL)$ is of weak type $(1,1)$ with respect to $\gamma,$ with a weak type constant which is uniform in $t>0.$ Then, since the sector $S_{\phi_p^{*}}$ approaches the right half-plane $S_{\pi/2}$ when $p\to 1^+,$ using the Marcinkiewicz interpolation theorem we see that the function $m$ is holomorphic (but not necessarily bounded) in $S_{\pi/2}$. An example of such an $m$ is a function of Laplace transform type in the sense of Stein \cite[pp. 58, 121]{topics}, i.e.\ $m(z)=z\int_0^{\infty}e^{-zt}\kappa(t)\,dt,$ with $\kappa\in L^{\infty}(\mathbb{R}_+,dt).$\footnote{Taking $\kappa(t)=e^{-it},$ so that $m(z)=z/(z+i),$ we see that these multipliers may be unbounded on $S_{\pi/2}.$}

Let now $A$ be a non-negative, self-adjoint operator defined on a space $L^2(Y,\mu),$ where $Y$ is equipped with a metric $\zeta$ such that $(Y,\zeta,\mu)$ is a space of homogeneous type, i.e.\ $\mu$ is a doubling measure. For simplicity we assume that $\mu(Y)=\infty,$ and that for all $x_2\in Y,$ the function $(0,\infty)\ni R\mapsto \mu(B_{\zeta}(x_2,R))$ is continuous and $\lim_{R\to 0}\mu(B_{\zeta}(x_2,R))=0.$ We further impose on $A$ the assumptions  \eqref{contra} and \eqref{noatomatzero} of Section \ref{sec:Prem}. Throughout this section we also assume that the heat semigroup $e^{-tA}$ has a kernel $e^{-tA}(x_2,y_2),$ $x_2,y_2\in Y,$ which is continuous on $\mathbb{R}^+\times Y\times Y,$ and satisfies the following Gaussian bounds.
\begin{equation}
\label{sec:OA,eq:gausbound} 0\le e^{-tA}(x_2,y_2)\leq \frac{C}{\mu (B(x_2,\sqrt{t}))}\exp(-c\zeta(x_2,y_2)^2\slash t),
\end{equation}
We also impose that for some $\delta>0,$ if $2\zeta(y_2,y'_2)\leq \zeta(x_2,y_2),$ then
\begin{equation}
\label{sec:OA,eq:heatlipsch}|e^{-tA}(x_2,y_2)-e^{-tA}(x_2,y'_2)|\leq \left(\frac{\zeta(y_2,y'_2)}{\sqrt{t}}\right)^{\delta}\frac{C}{\mu(B(x,\sqrt{t}))}\exp(-c \zeta(x_2,y_2)^2\slash t),
\end{equation}
while in general,
\begin{equation}
\label{sec:OA,eq:heatlipschngauss}|e^{-tA}(x_2,y_2)-e^{-tA}(x_2,y'_2)|\leq  \left(\frac{\zeta(y_2,y'_2)}{\sqrt{t}}\right)^{\delta}\frac{C}{\mu(B(x,\sqrt{t}))}.
\end{equation}

From \cite[Theorem 2.1]{Sik} (or rather its version for a single operator), it follows that, under \eqref{sec:OA,eq:gausbound}, the operator $A$ has a finite order Marcinkiewicz functional calculus on $L^p(Y,\mu),$ $1<p<\infty$. Examples of operators $A$ satisfying \eqref{sec:OA,eq:gausbound}, \eqref{sec:OA,eq:heatlipsch}, and \eqref{sec:OA,eq:heatlipschngauss} include, among others, the Laplacian $-\Delta$ and the harmonic oscillator $-\Delta+|x|^2$ on $L^2(\mathbb{R}^d,dx),$ or the Bessel operator $-\Delta-\sum_{j=1}^d \frac{2\alpha_j}{x_j}\partial_j$ (see \cite[Lemma 4.2]{dpw}).

Denote by $H^1=H^1(Y,\zeta,\mu)$ the atomic Hardy space in the sense of Coifman-Weiss \cite{CW}. More precisely, we say that a measurable function $b$ is an $H^1$-atom, if there exists a ball $B=B_{\zeta}\subseteq Y$, such that $\supp\, b \subset B,$ $\|b\|_{L^{\infty}(Y,\mu)}\leq 1/ \mu(B),$ and $\int_{Y}b(x_2)d\mu(x_2) =0.$ The space $H^1$ is defined as the set of all $g\in L^1(Y,\mu),$ which can be written as $g= \sum_{j=1}^{\infty} c_j b_j,$ where $b_j$ are atoms and $\sum_{j=1}^{\infty} |c_j|<\infty,$ $c_j\in\mathbb{C}.$ We equip $H^1$ with the norm
$
\|f\|_{H^1}=\inf \sum_{j=1}^{\infty} |c_j|,$
where the infimum runs over all absolutely summable $\{c_j\}_{j\in\mathbb{N}},$ for which $g= \sum_{j=1}^{\infty} c_j b_j,$ with $b_j$ being $H^1$-atoms. Note that from the very definition of $H^1$ we have $\|g\|_{L^1(Y,\mu)}\leq \|g\|_{H^1}.$

It can be shown that under \eqref{sec:OA,eq:gausbound}, \eqref{sec:OA,eq:heatlipsch}, and \eqref{sec:OA,eq:heatlipschngauss}, the space $$H^1_{max}=\{g\in L^1(Y,\mu)\colon \sup_{t>0}|e^{-tA}g|\in L^1(Y,\mu)\}$$ coincides with the atomic $H^1,$ i.e., there is a constant $C_{\mu}$ such that
\begin{equation}
\label{sec:OA,eq:maxchar}
C_{\mu}^{-1}\|g\|_{H^1} \leq \big\|\sup_{t>0}|e^{-tA}g|\big\|_{L^1(Y,\mu)}\leq C_{\mu} \|g\|_{H^1},\qquad g\in H^1(Y).
\end{equation}
The proof of \eqref{sec:OA,eq:maxchar} is similar to the proof of \cite[Proposition 4.1 and Lemma 4.3]{dpw}. The main trick is to replace the metric $\zeta$ with the measure distance (see \cite{CW})
$$\tilde{\zeta}(x_2,y_2)=\inf\{\mu(B)\colon B\textrm{ is a ball in Y},\, x_2,y_2\in B\},$$
 change the time $t$ via $$\mu(B(y,\sqrt t))=s,\qquad y\in Y,\quad t,\,s>0,$$ and apply Uchiyama's Theorem, see \cite[Corollary 1']{Uchi}. We omit the details. Note that by taking $r=e^{-t},$ the equation \eqref{sec:OA,eq:maxchar} can be restated as
\begin{equation}
\label{sec:OA,eq:maxcharr}
C_{\mu}^{-1}\|g\|_{H^1} \leq \big\|\sup_{0<r<1}|r^Ag|\big\|_{L^1(Y,\mu)}\leq C_{\mu} \|g\|_{H^1},\qquad g\in H^1(Y).
\end{equation}

For fixed $0<\varepsilon<1/2,$ define $M_{A,\varepsilon}(g)(x)=\int_Y\sup_{\varepsilon<r<1-\varepsilon}|r^A(x_2,y_2)||g(y_2)|\,d\mu(y_2).$ Then, a short reasoning using the Gaussian bound \eqref{sec:OA,eq:gausbound} and the doubling property of $\mu$ gives
\begin{equation}
\label{sec:OA,eq:maxcharrL1}
\left\|M_{A,\varepsilon}(g)\right\|_{L^1(Y,\mu)}\leq C_{\mu,\varepsilon} \|g\|_{L^1(Y,\mu)},\qquad g\in L^1(Y,\mu).
\end{equation}

Denote by $L^1_{\gamma}(H^1)$ the Banach space of those Borel measurable functions $f$ on $\mathbb{R}^d\times Y$ such that the norm
\begin{equation}\label{sec:OA,eq:L1H1norm}\|f\|_{L^1_{\gamma}(H^1)}=\int_{\mathbb{R}^d}\|f(x_1,\cdot)\|_{H^1}\,d\gamma(x),\end{equation}
is finite. In other words $L^1_{\gamma}(H^1)$ is the $L^1(\gamma)$ space of $H^1$-valued functions. Moreover, it is the closure of $$L^1_{\gamma}(\mathbb{R}^d)\odot H^1:=\bigg\{f\in L^1_{\gamma}(H^1)\colon f=\sum_{j}f_j^1\otimes f_j^2,\quad f_j^1 \in L^1_{\gamma}(\mathbb{R}^d),\, f_j^2\in H^1\bigg\}$$ in the norm given by \eqref{sec:OA,eq:L1H1norm}.

From now on in place of $\mL$ and $A$ we consider the tensor products $\mL \otimes I$ and $I \otimes A.$ Slightly abusing the notation we keep writing $\mL$ and $A$ for these operators. For the sake of brevity we write $L^p,$ $\|\cdot\|_{p}$ and $\|\cdot\|_{p \to p},$ instead of $L^p(\mathbb{R}^d \otimes Y, \gamma \otimes \mu),$ $\|\cdot\|_{L^p},$ and $\|\cdot\|_{L^p \to L^p},$ respectively. We shall also use the space $L^{1,\infty}:=L^{1,\infty}(\mathbb{R}^d\times Y,\gamma \otimes \mu),$ equipped with the quasinorm
 \begin{equation}\label{sec:OA,eq:weaknormgauss}\|f\|_{L^{1,\infty}}=\sup_{s>0}s (\gamma \otimes \mu)(\mathbb{R}^d\times Y\colon |f(x)|>s).\end{equation}
 Let $S$ be an operator which is of weak type $(1,1)$ with respect to $\gamma \otimes \mu.$ Then, $\|S\|_{L^1\to L^{1,\infty}}=\sup_{\|f\|_{1}=1}\|Sf\|_{L^{1,\infty}}$ is the best constant in its weak type $(1,1)$ inequality.

Let $m$ be a bounded function defined on $[0,\infty)\times \sigma(A),$ and let $m(\mL,A)$ be a joint spectral multiplier of $(\mL,A),$ as in \eqref{m(L)def}. Assume that for each $t>0,$ the operator $m(t\mL,A)$ is of weak type $(1,1)$ with respect to $\gamma\otimes \mu,$ with a weak type $(1,1)$ constant uniformly bounded with respect to $t.$ Then, from what was said before, we may conclude\footnote{At least in the case when $A$ has a discrete spectrum.} that for each fixed $a\in \sigma(A)$ the function $m(\cdot,a)$ has a holomorphic extension to the right half-plane. We limit ourselves to $m$ being of the following Laplace transform type:
\begin{equation} \label{sec:OA,eq:mult1} m(\la,a)=m_{\kappa}(\la,a):=\la\int_0^{\infty}e^{-\la t}e^{-a t}\kappa(t)\,dt,\qquad (\la,a)\in [0,\infty)\times \mathbb{R}_+,\end{equation}
with $\kappa\in L^{\infty}(\mathbb{R}_+,dt).$ In what follows we denote $\|\kappa\|_{\infty}=\|\kappa\|_{L^{\infty}(\mathbb{R}_+,dt)}.$

Observe that under the assumptions made on $A,$ the function $m_{\kappa}$ gives a well defined bounded operator $m_{\kappa}(\mL,A)$ on $L^2.$ Indeed, since $\chi_{\{a=0\}}(\mL,A)=0,$ we have
$$m_{\kappa}(\mL,A)=m_{\kappa}(\mL,A)\chi_{\{a>0\}}(\mL,A).$$ Moreover, $m_{\kappa}(0,a)=0$ for $a>0,$ and, consequently, the function $m_{\kappa}(\la,a)\chi_{\{a>0\}}$ is bounded on $[0,\infty)\times \mathbb{R}_+.$ Now, using the multivariate spectral theorem we see that $m_{\kappa}(\mL,A)$ is bounded on $L^2.$

The operator $m_{\kappa}(\mL,A)$ is also bounded on all $L^p$ spaces, $1<p<\infty.$ This follows from Corollary \ref{corHinf}. Moreover, we have $\|m\|_{p\to p}\leq C_p,$ with universal constants $C_p,$ $1<p<\infty.$

However, the following question is left open: is $m_{\kappa}(\mL,A)$ also of weak type $(1,1)?$ The main theorem of this section is a positive result in this direction.
\begin{thm}
\label{thm:OA}
Let $\mL$ be the Ornstein-Uhlenbeck operator on $L^2(\mathbb{R}^d,\gamma)$ and let $A$ be a non-negative self-adjoint operator on $L^2(Y,\zeta,\mu),$ satisfying all the assumptions of Section \ref{sec:Prem} and such that its heat kernel satisfies \eqref{sec:OA,eq:gausbound}, \eqref{sec:OA,eq:heatlipsch} and  \eqref{sec:OA,eq:heatlipschngauss}, as described in this section. Let $\kappa$ be a bounded function on $\mathbb{R}_+$ and let $m_{\kappa}$ be given by \eqref{sec:OA,eq:mult1}. Then the multiplier operator $m_{\kappa}(\mL,A)$ is bounded from $L^1_{\gamma}(H^1)$ to $L^{1,\infty}(\gamma\otimes \mu),$ i.e.\
\begin{equation}
\label{sec:OA,eq:weakH1type}
(\gamma\otimes \mu)(\{x\in \mathbb{R}^d\times Y\colon |m_{\kappa}(\mL,A)f(x)|>s\})\leq \frac{C_{d,\mu}\ki}{s}\|f\|_{L^1_{\gamma}(H^1)},\qquad s>0.
\end{equation}
\end{thm}
\begin{remark1}
Observe that $L^2 \cap L^1_{\gamma}(H^1)$ is dense in $L^1_{\gamma}(H^1).$ Thus, it is enough to prove \eqref{sec:OA,eq:weakH1type} for $f\in L^2 \cap L^1_{\gamma}(H^1).$
\end{remark1}
\begin{remark2}
Examples of multiplier operators of the form $m_{\kappa}(\mL,A)$ include the Riesz transforms $\mL(\mL+A)^{-1}$ (here $\kappa\equiv 1$) or the partial imaginary powers $\mL(\mL+A)^{-iu-1},$ $u\in \mathbb{R}$ (here $\kappa(t)=t^{iu}/\Gamma(iu+1)$). Note that since $I=\mL(\mL+A)^{-1}+A(\mL+A)^{-1},$ the boundedness of $\mL(\mL+A)^{-1}$ implies also the boundedness of $A(\mL+A)^{-1}$ from $L^1_{\gamma}(H^1)$ to $L^{1,\infty}(\gamma\otimes \mu).$  
\end{remark2}
Altogether, the proof of Theorem \ref{thm:OA} is rather long and technical, thus for the sake of the clarity of the presentation we do not provide all details. We use a decomposition of the kernel of the operator $T:=m_{\kappa}(\mL,A)$ into the global and local parts with respect to the Gaussian measure in the first variable. The local part will turn out to be of weak type $(1,1)$ (with respect to $\gamma \otimes \mu$) in the ordinary sense. For both the local and global parts we use ideas and some estimates from Garc\'ia-Cuerva, Mauceri, Sj\"ogren, and Torrea \cite{high} and \cite{laptype}.

Set $\kappa^{\varepsilon}=\kappa\chi_{[\varepsilon,1/\varepsilon]},$ $0<\varepsilon<1.$ Then, using the multivariate spectral theorem together with the fact that $A$ satisfies \eqref{noatomatzero}, we see that $\lim_{\varepsilon\to 0^+}m_{\kappa^{\varepsilon}}((\mL,A))=m_{\kappa}((\mL,A)),$ strongly in $L^2.$ Consequently, we also have convergence in the measure $\gamma\otimes \mu$. Since, clearly $\|\kappa^{\varepsilon}\|_{L^{\infty}(\mathbb{R}^+)}\leq \ki,$ it suffices to prove \eqref{sec:OA,eq:weakH1type} for $\kappa$ such that $\supp \kappa\subseteq[\varepsilon,1/\varepsilon].$\footnote{This reduction was suggested to us by Prof.\ Fulvio Ricci.} Thus, throughout the proof of Theorem \ref{thm:OA} we assume (often without further mention) that $\kappa$ is supported away from $0$ and $\infty.$ Additionally, the symbol $\lesssim$ denotes that the estimate is independent of $\kappa.$

In the proof of Theorem \ref{thm:OA} the variables with subscript $1,$ e.g.\ $x_1,y_1,$ are elements of $\mathbb{R}^d,$ while the variables with subscript $2,$ e.g.\ $x_2,y_2,$ are taken from $Y.$

We start with introducing some notation and terminology. Define $$L^{\infty}_c=\{f\in L^{\infty}\colon \supp f \textrm{ is compact}\}=\{f\in L^{\infty}(\mathbb{R}^d\times Y, \Lambda \otimes \mu)\colon \supp f \textrm{ is compact}\},$$
where $\Lambda$ is Lebesgue measure on $\mathbb{R}^d.$ Denoting $L^p(\mathbb{R}^d\times Y, \Lambda\otimes \mu):=L^p(\Lambda\otimes \mu),$ we see that for each $1\leq p<\infty,$ $L^{\infty}_c$ is a dense subspace of both $L^p$ and $L^p(\Lambda\otimes \mu).$ In particular, any operator which is bounded on $L^2$ or $L^2(\Lambda\otimes \mu)$ is well defined on $L^{\infty}_c.$ We also need the weak space $L^{1,\infty}(\Lambda\otimes \mu):=L^{1,\infty}(\mathbb{R}^d\times Y,\Lambda\otimes \mu)$ equipped with the quasinorm given by \eqref{sec:OA,eq:weaknormgauss} with $\gamma$ replaced by $\Lambda.$ An operator $S$ is of weak type $(1,1)$ precisely when
$$\|S\|_{L^1(\Lambda\otimes \mu) \to L^{1,\infty}(\Lambda\otimes \mu)}=\sup_{\|f\|_{L^1(\Lambda\otimes \mu)=1}}\|Sf\|_{L^{1,\infty}(\Lambda\otimes \mu)}<\infty.$$

Let $\eta$ be the product metric on $\mathbb{R}^d\times Y,$
\begin{equation}\label{sec:OA,eq:eta}\eta(x,y)=\max(|x_1-y_1|,\zeta(x_2,y_2)),\qquad x,y\in \mathbb{R}^d\times Y.
\end{equation}
Then it is not hard to see that the triple $(\mathbb{R}^d\times Y,\eta,\Lambda\otimes \mu)$ is a space of homogeneous type.
 \begin{defi}
 \label{defi:kernel}
 We say that a function $S(x,y)$ defined on the product $(\mathbb{R}^{d}\times Y)\times (\mathbb{R}^d\times Y)$ is a kernel of a linear operator $S$ defined on $L^{\infty}_c$ if, for every $f\in L^{\infty}_c$ and a.e.\ $x\in \mathbb{R}^d\times Y,$
$$Sf(x)=\int_{\mathbb{R}^d}\int_Y S(x,y)f(y)\,d\mu(y_2)\,dy_1.$$
\end{defi}
\begin{remark1}
We do not restrict to $x\not\in \supp f;$ the operators we consider later on are well defined in terms of their kernels for all $x.$ This is true because of the assumption that $\kappa$ is supported away from $0$ and $\infty.$
\end{remark1}
\begin{remark2}
The reader should keep in mind that the inner integral defining $Sf(x)$ is taken with respect to the Lebesgue measure $dy_1$ rather than the Gaussian measure $d\gamma(y_1).$ The reason for this convention is the form of Mehler's formula we use, see \eqref{sec:OA,eq:MehlformOUint}.
\end{remark2}

Let $\mM_r(x_1,y_1),$ $x_1,y_1 \in \mathbb{R}^d,$ $0<r<1,$  denote Mehler's kernel in $\mathbb{R}^d,$ i.e.\ the kernel of the operator $r^{\mL}=e^{-t\mL},$ with $r=e^{-t}.$ It is well known that, for $0<r<1,$
\begin{equation}
\label{sec:OA,eq:MehlformOU}
\mM_r(x_1,y_1)=\pi^{-d/2}(1-r^2)^{-d/2}\exp\bigg(-\frac{|rx_1-y_1|^2}{1-r^2}\bigg),\qquad x_1,y_1\in\mathbb{R}^d.
\end{equation}
and that, for all $g\in L^p(\mathbb{R}^d,\gamma)$ with $1\leq p\leq \infty,$
\begin{equation}\label{sec:OA,eq:MehlformOUint} r^{\mL}g(x_1)=\int_{\mathbb{R}^d}\mM_r(x_1,y_1)g(y_1)\,dy_1,\qquad x_1\in \mathbb{R}^d.\end{equation}
In particular, using \eqref{sec:OA,eq:MehlformOUint} it can be deduced that $\{e^{-t\mL}\}_{t>0}$ satisfies the contractivity condition \eqref{contra}. Additionally, a short computation using \eqref{sec:OA,eq:MehlformOU} gives
\begin{equation}
\label{sec:OA,eq:comM}
\begin{split}
\partial_r\,\mM_r(x_1,y_1)= &\pi^{-d/2}\left(dr-2r\frac{|rx_1-y_1|^2}{1-r^2}-\langle rx_1-y_1,x_1\rangle\right)(1-r^2)^{-d/2-1}\\
&\times \exp\bigg(-\frac{|rx_1-y_1|^2}{1-r^2}\bigg).
\end{split}
\end{equation}
From the above we see that, if $\varepsilon<r<1-\varepsilon,$ for some $0<\varepsilon<1/2,$ then
\begin{equation}
\label{sec:OA,eq:comMestfar01}
|\partial_r\,\mM_r(x_1,y_1)|\lesssim C_{\varepsilon}(1+|x_1|).
\end{equation}

Note that, since $\kappa$ is a bounded function supported away from $0$ and infinity, the function $\kappa_{\log}(r)=\kappa(-\log r),$ $0<r<1,$ is also bounded and supported away from $0$ and $1,$ say in an interval $[\varepsilon,1-\varepsilon],$ $0<\varepsilon<1/2.$ Moreover, we have $\|\kappa\|_{L^{\infty}((0,\infty),dt)}=\|\kappa_{\log}\|_{L^{\infty}((0,1),dt)}.$ In what follows, slightly abusing the notation, we keep the symbol $\kappa$ for the function $\kappa_{\log}.$

The change of variable $r=e^{-t}$ leads to the formal equality
$$T=\int_{0}^{1}\kappa(r)\mL r^{\mL}r^A \,\frac{dr}{r}=\int_{0}^{1}\kappa(r)\partial_r r^{\mL}r^A \,dr.$$ Suggested by the above we define the kernel
\begin{equation*}
K(x,y)=\int_0^{1}\partial_r \mM_r(x_1,y_1)\,r^A(x_2,y_2)\kappa(r)\,dr,\qquad x_1,y_1\in\mathbb{R}^d,\quad x_2,y_2\in Y,
\end{equation*}
with $r^A(x_2,y_2)=e^{(\log r) A}(x_2,y_2).$ Then we have.

\begin{lem}
\label{lem:kerKopT}
The function $K$ is a kernel of $T$ in the sense of Definition \ref{defi:kernel}.
\end{lem}
\begin{proof}[Proof (sketch)]
It is enough to show that for $f,h \in L_c^{\infty}$ we have
\begin{equation}
\label{sec:OA,eq:itisenough}
\langle Tf, h\rangle=\int_{\mathbb{R}^d\times Y}\int_{\mathbb{R}^d\times Y}K(x,y)f(y)h(x)\,d(\Lambda\otimes \mu)(y)\,d(\gamma\otimes \mu)(x).\end{equation}

From the multivariate spectral theorem together with Fubini's theorem we see that
\begin{equation}
\label{sec:OA,eq:multFubin}
\langle m(\mL,A)f, g\rangle_{L^2}=\int_0^1\kappa(r) \langle \mL r^{\mL-1} r^A f, h\rangle_{L^2}\,dr, \qquad f,\, h\in L^2.\end{equation}
Now, by the multivariate spectral theorem $\mL r^{\mL-1} (r^A f)=(\partial_r r^{\mL}) (r^A f),$ where on right hand side we have the Fr\'echet derivative in $L^2.$ Thus, $\langle \mL r^{\mL-1} r^A f, h\rangle_{L^2}$ is the limit (as $\delta\to 0$) of \begin{align}\label{sec:OA,eq:limandint}
&\delta^{-1}\langle((r+\delta)^{\mL}-r^{\mL}) r^A f, h\rangle_{L^2}\\
&\nonumber=
\int_{\mathbb{R}^d\times Y}\int_{\mathbb{R}^d\times Y}\frac{\mM_{r+\delta}(x_1,y_1)-\mM_{r}(x_1,y_1)}{\delta}r^{A}(x_2,y_2)f(y)\, h(x)\,d(\Lambda\otimes \mu)(y) \,d(\gamma\otimes\mu)(x).\end{align}

Since $f,g\in L_c^{\infty},$ using \eqref{sec:OA,eq:maxcharrL1}, \eqref{sec:OA,eq:comMestfar01}, and the dominated convergence theorem we justify taking the limit inside the integral in \eqref{sec:OA,eq:limandint} and obtain
$$\langle \mL r^{\mL-1} r^A f, h\rangle_{L^2}=\int_{\mathbb{R}^d\times Y}\int_{\mathbb{R}^d\times Y}\partial_r\mM_{r}(x_1,y_1)r^{A}(x_2,y_2)f(y)\, h(x)\,d(\Lambda\otimes \mu)(y) \,d(\gamma\otimes\mu)(x).$$
Plugging the above formula into \eqref{sec:OA,eq:multFubin}, and using Fubini's theorem (which is allowed by \eqref{sec:OA,eq:maxcharrL1}, \eqref{sec:OA,eq:comMestfar01} and the fact that $\supp \kappa \subseteq [\varepsilon, 1-\varepsilon]$), we arrive at \eqref{sec:OA,eq:itisenough}, as desired.
\end{proof}
Let $N_s,$ $s>0,$ be given by
\begin{equation*}
N_s=\big\{(x_1,y_1)\in\mathbb{R}^d\times \mathbb{R}^d\colon |x_1-y_1|\leq \frac{s}{1+|x_1|+|y_1|}\big\}.
\end{equation*}
We call $N_s$ the local region with respect to the Gaussian measure $\gamma$ on $\mathbb{R}^d.$ This set (or its close variant) is very useful when studying maximal operators or multipliers for $\mL.$ After being applied by Sj\"ogren in \cite{Sj1}, it was used in \cite{funccalOu}, \cite{high}, \cite{laptype}, and \cite{sharp}, among others.

The local and global parts of the operator $T$ are defined, for $f\in L_c^{\infty},$ by
\begin{equation}
\label{sec:OA,eq:glob1}
T^{glob}f(x)=\int_{\mathbb{R}^d}\int_Y (1-\chi_{N_2}(x_1,y_1))K(x,y)f(y)\,d\mu(y_2)\,dy_1,
\end{equation}
and
\begin{equation*}
T^{loc}f(x)=Tf(x)-T^{glob}f(x),
\end{equation*}
respectively. The estimates from Proposition \ref{prop:proglob} demonstrate that the integral \eqref{sec:OA,eq:glob1} defining $T^{glob}$ is absolutely convergent for a.e.\ $x,$ whenever $f\in L^1.$

Note that the cut-off considered in \eqref{sec:OA,eq:glob1} is the rough one from \cite[p.\ 385]{high} (though only with respect to $x_1,y_1$) rather than the smooth one from \cite[p.\ 288]{laptype}. In our case, using a smooth cut-off with respect to $\mathbb{R}^d$ does not simplify the proofs. That is because, even a smooth cut-off with respect to $x_1,y_1$ may not preserve a Calder\'on-Zygmund kernel in the full variables $(x,y).$ Moreover, the rough cut-off has the advantage that $(T^{loc})^{loc}=T^{loc}.$

We begin with proving the desired weak type $(1,1)$ property for $T^{glob}.$ Since
$$T^{glob}f(x)=\int_{0}^{1} \int_{\mathbb{R}^d}\partial_r \mM_r(x_1,y_1)\chi_{N_2^c}(x_1,y_1)\,r^A(f(y_1,\cdot))(x_2)\,dy_1\,\kappa(r)\,dr$$
and $\supp \kappa \subseteq [\varepsilon,1-\varepsilon]$ we have
\begin{equation}
\label{sec:OA,eq:glob2}
\begin{split}
&|T^{glob}f(x)|\leq \|\kappa\|_{\infty}\int_{\varepsilon}^{1-\varepsilon} \int_{\mathbb{R}^d}|\partial_r \mM_r(x_1,y_1)|\chi_{N_2^c}(x_1,y_1)|r^A(f(y_1,\cdot))(x_2)|\,dy_1\,dr\\
&\leq \|\kappa\|_{\infty} \int_0^{1} \int_{\mathbb{R}^d}|\partial_r \mM_r(x_1,y_1)|\chi_{N_2^c}(x_1,y_1)\sup_{\varepsilon<r<1-\varepsilon}|r^A(f(y_1,\cdot))(x_2)|\,dy_1\,dr\\
&:=\|\kappa\|_{\infty}T_*^{glob}f(x).
\end{split}
\end{equation}
Moreover, the following proposition holds.
\begin{pro}
\label{prop:proglob}
The operator $T_*^{glob}$ is well defined on $L^1$ and bounded from $L^1_{\gamma}(H^1)$ to $L^{1,\infty}(\gamma\otimes \mu),$ with a bound independent of $0<\varepsilon<1/2.$ Thus, $T^{glob}$ is also well defined on $L^1$ and we have
\begin{equation*}
(\gamma\otimes \mu)(\{x\in \mathbb{R}^d\times Y\colon |T^{glob}f(x)|>s\})\leq \frac{C_{d,\mu}\ki}{s}\|f\|_{L^1_{\gamma}(H^1)},\qquad s>0.
\end{equation*}
\end{pro}
\begin{proof}
By \eqref{sec:OA,eq:glob2} it clearly suffices to focus on $T_*^{glob}.$

Using the finite sign change argument, i.e.\ the inequality (2.3) from the proof of \cite[Lemma 2.1]{laptype}, we see that
$$T_*^{glob}f(x)\leq C \int_{\mathbb{R}^d}\sup_{0<r<1}\mM_r(x_1,y_1)\chi_{N_2^c}(x_1,y_1)f^*_2(y_1,x_2)\,dy_1, $$
where $f^*_2(x_1,x_2)=\sup_{\varepsilon<r<1-\varepsilon}|r^A(f(x_1,\cdot))(x_2)|.$ Moreover, from \cite[Theorem 3.8]{laptype} and \cite[Lemma 2]{Sj1} it follows that the operator $$L^1(\mathbb{R}^d,\gamma)\ni g\mapsto \int_{\mathbb{R}^d}\sup_{0<r<1}\mM_r(x_1,y_1)\chi_{N_2^c}(x_1,y_1)|g|(y_1)\,dy_1:=T^{*}_1 g(x_1),$$ is of weak type $(1,1)$ with respect to $\gamma.$
Hence, using Fubini's theorem we have
\begin{align}
&\nonumber(\gamma\otimes \mu)(\{x\in \mathbb{R}^d\times Y\colon |T_*^{glob}f(x)|>s\})= \int_{Y}\gamma(\{x_1\in\mathbb{R}^d\colon |T_*^{glob}f(x)|>s\})\,d\mu(x_2)\\ \nonumber
&\leq \int_{Y}\gamma(\{x_1\in\mathbb{R}^d\colon |T^{*}_1(f^*_2(\cdot,x_2))(x_1)|>s\})\,d\mu(x_2)\leq  \int_{Y}\frac{C_d}{s}\int_{\mathbb{R}^d}f^*_2(x_1,x_2)\,d\gamma(x_1)\,d\mu(x_2)\\
&=\frac{C_d}{s} \int_{\mathbb{R}^d}\int_{Y}\sup_{\varepsilon<r<1-\varepsilon}|r^A(f(x_1,\cdot))(x_2)|\,d\mu(x_2)\,d\gamma(x_1). \label{sec:OA,eq:FubestTglob}
\end{align}
Now, from \eqref{sec:OA,eq:maxcharrL1} we see that, for each fixed $0<\varepsilon<1/2,$ the operator $T_*^{glob}$ is of weak type $(1,1)$ with respect to $\gamma\otimes \mu;$ in particular, it is well defined for $f\in L^1.$ Finally, using \eqref{sec:OA,eq:FubestTglob} and \eqref{sec:OA,eq:maxcharr}, we obtain the (independent of $\varepsilon$) boundedness of $T_*^{glob}$ from $L^1_{\gamma}(H^1)$ to $L^{1,\infty}(\gamma\otimes \mu).$
\end{proof}

Now we turn to the local part $T^{loc}.$ As we already mentioned, $T^{loc}$ turns out to be of (classical) weak type $(1,1)$ with respect to $\gamma\otimes \mu.$
\begin{pro}
\label{thm:tlocalpart}
The operator $T^{loc}$ is of weak type $(1,1)$ with respect to $\gamma\otimes \mu,$ and $\|T^{loc}\|_{L^1\to L^{1,\infty}}\lesssim \ki.$ Thus, $T^{loc}$ is also bounded from $L^1_{\gamma}(H^1)$ to $L^{1,\infty}(\gamma\otimes \mu),$ and
\begin{equation*}
(\gamma\otimes \mu)(\{x\in \mathbb{R}^d\times Y\colon |T^{loc}f(x)|>s\})\leq \frac{C_{d,\mu}\ki}{s}\|f\|_{L^1_{\gamma}(H^1)},\qquad s>0.
\end{equation*}
\end{pro}
From now on we focus on the proof of Proposition \ref{thm:tlocalpart}. The key ingredient is a comparison (in the local region) of the kernel $K$ with a certain convolution kernel $\tilde{K}$ in the variables $(x_1,y_1),$ i.e.\ depending on $(x_1-y_1,x_2,y_2).$ We also heavily exploit the fact that in the local region $N_2$ the measure $\gamma\otimes \mu$ is comparable with $\Leb\otimes \mu.$

For further reference we restate \cite[Lemma 3.1]{laptype}. The first five items of Lemma \ref{lem:Lem31laptype} are exactly items i)-v) from \cite[Lemma 3.1]{laptype}, item vi) is \cite[eq.\ (3.2) p.\ 289]{laptype}, while item vii) is \cite[eq.\ (3.3) p.\ 289]{laptype}.
\begin{lem}
\label{lem:Lem31laptype}
There exists a family of balls on $\mathbb{R}^d$
$$B_j=B\left(x_1^j,\frac{1}{20(1+|x_1^j|)}\right),$$
such that:
\begin{enumerate}[i)]
\item the family $\{B_j\colon j\in \mathbb{N}\}$ covers $\mathbb{R}^d$;
\item the balls $\{\frac14 B_j \colon j\in \mathbb{N}\}$ are pairwise disjoint;
\item for any $\beta > 0$, the family $\{\beta B_j \colon j\in \mathbb{N}\}$ has bounded overlap, i.e.;
$\sup \sum_j \chi_{\beta B_j}(x_1)\leq C$;
\item $B_j \times 4B_j \subseteq N_1$ for all $j\in \mathbb{N}$;
\item if $x_1 \in B_j,$ then $B(x_1,\frac{1}{20(1+|x_1|)}) \subseteq 4B_j$;
\item for any measurable $V\subseteq 4B_j,$ we have $\gamma(V)\approx e^{-|x_1^j|^2}\Lambda(V);$
\item $N_{1/7}\subseteq \bigcup_j B_j \times 4B_j \subseteq N_2.$
\end{enumerate}
\end{lem}

The next lemma we need is a two variable version of \cite[Lemma 3.3]{laptype} (see also the following remark). The proof is based on Lemma \ref{lem:Lem31laptype} and proceeds as in \cite{laptype}. We omit the details, as the only ingredient that needs to be added is an appropriate use of Fubini's theorem. In Lemma \ref{lem:Lem33laptype} by $\nu$ we denote one of the measures $\gamma$ or $\Lambda.$
\begin{lem}
\label{lem:Lem33laptype}
Let $S$ be a linear operator defined on $L_c^{\infty}$ and
set
$$S_1f(x)=\sum_j \chi_{B_j}(x_1)S(\chi_{4B_j}(y_1)f)(x),$$
where $B_j$ is the family of balls from Lemma \ref{lem:Lem31laptype}. We have the following:
\begin{enumerate}[i)]
\item If $S$ is of weak type $(1,1)$ with respect to the measure $\nu\otimes \mu,$ then $S_1$ is of weak type $(1,1)$ with respect to both $\gamma \otimes \mu$ and $\Lambda \otimes \mu;$ moreover, $$\|S_1\|_{L^1\to L^{1,\infty}}+\|S_1\|_{L^1(\Lambda\otimes \mu)\to L^{1,\infty}(\Lambda\otimes \mu)}\lesssim \|S\|_{L^1(\nu\otimes \mu)\to L^{1,\infty}(\nu\otimes \mu)}.$$
\item If $S$ is bounded on $L^p(\mathbb{R}^d\times Y,\nu\otimes \mu),$ for some $1<p<\infty,$ then $S_1$ is bounded on both $L^p$ and $L^p(\Lambda \otimes\mu);$ moreover,
    $$\|S_1\|_{p\to p}+\|S_1\|_{L^p(\Lambda\otimes \mu)\to L^p(\Lambda\otimes \mu)}\lesssim \|S\|_{L^p(\nu\otimes \mu)\to L^p(\nu\otimes \mu)}.$$
\end{enumerate}
\end{lem}

We proceed with the proof of Proposition \ref{thm:tlocalpart}. Decompose $T=D+\tilde{T},$ where,
\begin{align*}
Df&=\int_{0}^{1}\kappa(r)\,\partial_r [r^{\mathcal{L}}-e^{\frac{1}{4}(1-r^2)\Delta}]\,r^{A}f\,dr,\\
\tilde{T}f&=\int_{0}^{1}\kappa(r)\,\partial_r e^{\frac{1}{4}(1-r^2)\Delta}\, r^{A}f\,dr,
\end{align*}
with $\Delta$ being the self-adjoint extension of the Laplacian on $L^2(\mathbb{R}^d,\Lambda).$ Observe that, by the multivariate spectral theorem applied to the system $(-\Delta,A),$ the operator $\tilde{T}$ is bounded on $L^2(\Lambda\otimes\mu).$ Consequently, $\tilde{T}$ and thus also $D=T-\tilde{T},$ are both well defined on $L_c^{\infty}.$

We start with considering the operator $\tilde{T}.$ First we demonstrate that $$\tilde{T}=\int_{0}^{1}\kappa(r)\,\partial_r e^{\frac{1}{4}(1-r^2)\Delta}\, r^{A}\,dr$$ is a Calder\'on-Zygmund operator on the space of homogeneous type $(\mathbb{R}^d\times Y, \eta, \Leb\otimes \mu);$ recall that $\eta$ is defined by \eqref{sec:OA,eq:eta}. In what follows $\tilde{K}$ is given by
\begin{equation*}
\tilde{K}(x,y)=\int_{0}^{1}\kappa(r)\,\partial_r\mW_r(x_1-y_1)\, r^{A}(x_2,y_2)\,dr,
\end{equation*}
with
\begin{equation}
\label{sec:OA,eq:mWrdef}
\mW_r(x_1,y_1)=\pi^{-d/2}(1-r^2)^{-d/2}\exp\bigg(-\frac{|x_1-y_1|^2}{1-r^2}\bigg).\end{equation}
In the proof of Lemma \ref{lem:calzyg} we often use the following simple bound
$$\int_0^{\infty}t^{-\alpha}\exp(-\beta t^{-1})\,dt\lesssim \beta^{-\alpha+1},\qquad \alpha>1,\quad\beta >0,$$
cf.\ \cite[Lemma 1.1]{StTor}, without further mention.

\begin{lem}
\label{lem:calzyg}
The operator $\tilde{T}$ is a Calder\'on-Zygmund operator associated with the kernel $\tilde{K}.$ More precisely, $\tilde{T}$ is bounded on $L^2(\mathbb{R}^d\times Y, \eta, \Leb\otimes \mu),$ with
\begin{equation}
\label{sec:OA,eq:L2tilT}
\|\tilde{T}\|_{L^2( \Leb\otimes \mu)\to L^2( \Leb\otimes \mu)}\lesssim \ki,
\end{equation}
and its kernel satisfies standard Calder\'on-Zygmund estimates, i.e.\ the growth estimate
\begin{equation}
\label{sec:OA,eq:growth}
|\tilde{K}(x,y)|\lesssim \frac{\ki}{(\Leb\otimes \mu)(B(x,\eta(x,y)))},\qquad x\neq y,
\end{equation}
and, for some $\delta>0,$ the smoothness estimate
\begin{align}
\label{sec:OA,eq:smooth}
|\tilde{K}(x,y)-\tilde{K}(x,y')|\lesssim \left(\frac{\eta(y,y')}{\eta(x,y)}\right)^{\delta}\frac{\ki}{(\Leb\otimes \mu)(B(x,\eta(x,y)))},\qquad 2\eta(y,y')\leq \eta(x,y).
\end{align}
Consequently $\tilde{T}$ is of weak type $(1,1)$ with respect to $\Leb\otimes \mu,$ and
$$(\Leb\otimes \mu)(x\in \mathbb{R}^d\times Y\colon |\tilde{T}f(x)|>s)\leq \frac{C_{d,\mu}\ki}{s}\|f\|_{L^1(Y,\mu)},\qquad s >0.$$
\end{lem}
\begin{proof}
As we have already remarked, by spectral theory $\tilde{T}$ is bounded on $L^2(\Leb\otimes \mu),$ and we easily see that \eqref{sec:OA,eq:L2tilT} holds. Additionally, an argument similar to the one used in the proof of Lemma \ref{lem:kerKopT} shows that $\tilde{T}$ is associated with the kernel $\tilde{K}$ even in the sense of Definition \ref{defi:kernel}.

We now pass to the proofs of the growth and smoothness estimates and start with demonstrating \eqref{sec:OA,eq:growth}. An easy calculation shows that
\begin{equation}
\partial_r\mW_r(x_1-y_1)=\pi^{-d/2}r(1-r^2)^{-d/2-1}\exp\bigg(-\frac{|x_1-y_1|^2}{1-r^2}\bigg)\left[d-2\frac{|x_1-y_1|^2} {1-r^2}\right]. \label{sec:OA,eq:comker}
\end{equation}
Hence, we have for $x_1,y_1 \in\mathbb{R}^d$
\begin{equation}
\label{sec:OA,eq:comestt}
|\partial_r\mW_r(x_1-y_1)|\lesssim r(1-r)^{-d/2-1}\exp\bigg(-\frac{|x_1-y_1|^2}{4(1-r)}\bigg),\qquad 0<r<1.
\end{equation}
For further use we remark that the above bound implies
\begin{equation}
\label{sec:OA,eq:comesttint}
\int_0^{1}|\partial_r\mW_r(x_1-y_1)|\,dr \lesssim |x_1-y_1|^{-d},\qquad x_1,y_1\in\mathbb{R}^d,\quad x_1\neq y_1.
\end{equation}

From \eqref{sec:OA,eq:comestt} we see that
\begin{equation}
\label{sec:OA,eq:comesttt}
|[\partial_r\mW_r(x_1-y_1)]_{r=e^{-t}}|
\leq
\left\{
\begin{array}{rl}
 &C t^{-d/2-1}\exp\bigg(-\frac{|x_1-y_1|^2}{ct}\bigg),\qquad t\leq 1, \\
  & C  e^{-t}\exp\bigg(-c|x_1-y_1|^2\bigg),\qquad t>1.
 \end{array} \right.
\end{equation}
Thus, coming back to the variable $t=-\log r$ and then using \eqref{sec:OA,eq:gausbound}, we arrive at
\begin{align*}
&|\tilde{K}(x,y)|\lesssim \ki \int_0^{\infty}t^{-d/2-1}\exp\bigg(-\frac{|x_1-y_1|^2}{ct}\bigg)\frac{1}{\mu(B(x_2,\sqrt{t}))}\exp\bigg(-\frac{\zeta^2(x_2,y_2)}{ct}\bigg)\,dt.
\end{align*}
A standard argument using the doubling property of $\mu$ (cf. \eqref{sec:OA,eq:double}) shows that we can further estimate
\begin{align*}
&|\tilde{K}(x,y)|\lesssim \frac{\ki}{\mu(B(x_2,\eta(x,y)))} \int_0^{\infty}t^{-d/2-1}\exp\bigg(-\frac{\eta^2(x,y)}{2ct}\bigg)\,dt.
\end{align*}
The last integral is bounded by a constant times $\eta^{d}(x,y),$ which equals $C_d \Leb(B_{|\cdot|}(x_1,\eta(x,y))).$ Thus, \eqref{sec:OA,eq:growth} follows once we note that
$$ \frac{1}{\Leb(B_{|\cdot|}(x_1,\eta(x,y))\mu(B_{\zeta}(x_2,\eta(x,y)))}=\frac{1}{(\Leb\otimes \mu) (B(x,\eta(x,y)))}.$$

We now focus on the smoothness estimate \eqref{sec:OA,eq:smooth}, which is enough to obtain the desired weak type $(1,1)$ property of $\tilde{T}.$ We decompose the difference in \eqref{sec:OA,eq:smooth} as
\begin{equation*}
\tilde{K}(x,y)-\tilde{K}(x,y')=[\tilde{K}(x,y)-\tilde{K}(x,y'_1,y_2)]+[\tilde{K}(x,y'_1,y_2)-\tilde{K}(x,y')]\equiv I_1+I_2.
\end{equation*}
Till the end of the proof of \eqref{sec:OA,eq:smooth} we assume $\eta(x,y)\geq 2 \eta(y,y'),$ so that $\eta(x,y)\approx \eta(x,y').$

We start with estimating $I_2$ and consider two cases. First, let $|x_1-y_1|\leq \zeta(x_2,y_2).$ Then, $\eta(x,y)=\zeta(x_2,y_2)\geq 2\eta(y,y')\geq 2\zeta(y_2,y'_2)$ and consequently, $\zeta(x_2,y'_2)\approx \zeta(x_2,y_2).$ Now, coming back to the variable $t=-\log r$ and using \eqref{sec:OA,eq:comesttt} we have
\begin{align*}
|I_2|\lesssim \ki \int_0^{\infty} t^{-d/2-1}\exp\bigg(-\frac{|x_1-y'_1|^2}{ct}\bigg)\big|e^{-tA}(x_2,y_2)-e^{-tA}(x_2,y'_2)\big|\,dt.
\end{align*}
Hence, from \eqref{sec:OA,eq:heatlipsch} it follows that
\begin{align}
\label{sec:OA,eq:quan}
|I_2|\lesssim \ki \zeta(y_2,y'_2)^{\delta}\int_0^{\infty}t^{-d/2-1-\delta/2}\frac{1}{\mu(B(x_2,\sqrt {t})}\exp\bigg(-\frac{\eta^2(x,y)}{ct}\bigg)\,dt
\end{align}
Using the doubling property of $\mu$ it is not hard to see that
\begin{equation}
\label{sec:OA,eq:double}
\frac{1}{\mu(B(x_2,\sqrt t)}\exp\bigg(-\frac{\eta^2(x,y)}{ct}\bigg)\lesssim \frac{1}{\mu(B(x_2,\eta(x,y)))}\exp\bigg(-\frac{\eta^2(x,y)}{2ct}\bigg),
\end{equation}
and consequently,
\begin{align*}
|I_2|&\lesssim \ki \zeta(y_2,y'_2)^{\delta}\frac{1}{\mu(B(x_2,\eta(x,y)))}\int_0^{\infty}t^{-d/2-1-\delta/2}\exp\bigg(-\frac{\eta^2(x,y)}{3ct}\bigg)\,dt
\\&\lesssim \ki\zeta(y_2,y'_2)^{\delta}\eta(x,y)^{-\delta}\frac{1}{\mu(B(x_2,\eta(x,y)))}(\eta^2(x,y))^{-d/2},
\end{align*}
thus proving that
\begin{equation}
\label{sec:OA,eq:I2}
|I_2|\lesssim \bigg(\frac{\zeta(y_2,y'_2)}{\eta(x,y)}\bigg)^{\delta}\frac{\ki}{(\Leb\otimes\mu)(B(x,\eta(x,y)))}.
\end{equation}

Assume now that $\zeta(x_2,y_2)\leq |x_1-y_1|.$ In this case $\eta(x,y)=|x_1-y_1|>2\eta(y,y')\geq 2|y_1-y'_1|,$ so that $|x_1-y_1|\approx |x_1-y'_1|.$ Hence, proceeding similarly as in the previous case (this time we use \eqref{sec:OA,eq:heatlipschngauss} instead of \eqref{sec:OA,eq:heatlipsch}), we obtain
\begin{align*}
|I_2|\lesssim \ki \zeta(y_2,y'_2)^{\delta} \int_0^{\infty}t^{-d/2-1-\delta/2}\frac{1}{\mu(B(x_2,\sqrt{t}))}\exp\bigg(-\frac{\eta^2(x,y)}{ct}\bigg)\,dt.
\end{align*}
The latter quantity has already appeared in \eqref{sec:OA,eq:quan} and has been estimated by the right hand side of \eqref{sec:OA,eq:I2}.

Now we pass to $I_1.$ A short computation based on \eqref{sec:OA,eq:comker} gives
\begin{equation*}
\pi^{d/2}\partial_{z_j}\partial_r \mW_r (z)=-2r(1-r^2)^{-d/2-2}z_j\left(d+2-2\frac{|z|^2}{1-r^2}\right)\exp\bigg(\frac{-|z|^2}{1-r^2}\bigg),\qquad z\in \mathbb{R}^d.
\end{equation*}
From the above inequality it is easy to see that
$$|\partial_{z_j}\partial_r\mW_{r}(z)|\lesssim r(1-r)^{-d/2-3/2}\exp\bigg(\frac{-|z|^2}{2(1-r^2)}\bigg),$$
and consequently, after the change of variable $e^{-t}=r,$
$$|\partial_{z_j}\partial_r\mW_{r}(z)\big|_{r=e^{-t}}|\lesssim t^{-d/2-3/2}\exp\bigg(\frac{-|z|^2}{ct}\bigg),\qquad 0<t<\infty.$$
Hence, from the mean value theorem it follows that for $|x_1-y_1|\geq 2 |y_1-y'_1|,$
\begin{align}
\label{sec:OA,eq:lipmW}
&|\left[\partial_r\mW_{r}(x_1-y_1)-\partial_r\mW_{r}(x_1-y'_1)\right]_{r=e^{-t}}|\lesssim \frac{|y_1-y'_1|}{\sqrt{t}} \, t^{-d/2-1}\exp\bigg(\frac{-|x_1-y_1|^2}{ct}\bigg)\,
\end{align}
while for arbitrary $x_1,y_1,$
\begin{align}
\label{sec:OA,eq:lipngaussmW}
&|\left[\partial_r\mW_{r}(x_1-y_1)-\partial_r\mW_{r}(x_1-y'_1)\right]_{r=e^{-t}}|\lesssim \frac{|y_1-y'_1|}{\sqrt{t}}\, t^{-d/2-1}.
\end{align}
Moreover, at the cost of a constant in the exponent, the expression $|y_1-y'_1|/\sqrt{t}$ from the right hand sides of \eqref{sec:OA,eq:lipmW} and \eqref{sec:OA,eq:lipngaussmW} can be replaced by $(|y_1-y'_1|t^{-1/2})^{\delta},$ for arbitrary $0<\delta \leq 1.$
If $|y_1-y'_1|\leq \sqrt{t},$ this is a consequence of \eqref{sec:OA,eq:lipmW} and \eqref{sec:OA,eq:lipngaussmW}, while if $|y_1-y'_1|\geq \sqrt{t}$ it can be deduced from \eqref{sec:OA,eq:lipmW} and \eqref{sec:OA,eq:comesttt}. Similarly as it was done for $I_2,$ to estimate $I_1$ we consider two cases.

Assume first $|x_1-y_1|\geq\zeta(x_2,y_2),$ so that $\eta(x,y)=|x_1-y_1|> 2\eta(y,y')\geq 2|y_1-y'_1|$ and $|x_1-y_1|\approx |x_1-y'_1|.$ Therefore, using \eqref{sec:OA,eq:gausbound} and the version of \eqref{sec:OA,eq:lipmW} with $(|y_1-y_1'|t^{-1/2})^{\delta}$ in place of $|y_1-y_1'|/\sqrt{t},$ we obtain
\begin{align*}
|I_1|\lesssim \ki |y_1-y'_1|^{\delta}\int_0^{\infty} t^{-d/2-1-\delta/2}\frac{1}{\mu(B(x_2,\sqrt{t}))}\exp\bigg(-\frac{\eta^2(x,y)}{ct}\bigg)\,dt.
\end{align*}
Almost the same quantity appeared already in \eqref{sec:OA,eq:quan}, thus employing once again previous techniques, we end up with
\begin{equation}
\label{sec:OA,eq:I1}
|I_1|\lesssim \bigg(\frac{|y_1-y'_1|}{\eta(x,y)}\bigg)^{\delta}\frac{\ki}{(\Leb\otimes \mu)(B(x,\eta(x,y)))}.
\end{equation}

Assume now that $|x_1-y_1|<\zeta(x_2,y_2),$ so that $\eta(x,y)=\zeta(x_2,y_2)>2\eta(y,y')\geq 2\zeta(y_2,y'_2)$ and $\zeta(x_2,y_2)\approx \zeta(x_2,y'_2).$ This time, from \eqref{sec:OA,eq:gausbound} and the $\delta$ version of \eqref{sec:OA,eq:lipngaussmW} we have
\begin{align*}
|I_1|\lesssim \ki |y_1-y'_1|^{\delta}\int_0^{\infty} t^{-d/2-1-\delta/2}\frac{1}{\mu(B(x_2,\sqrt{t}))}\exp\bigg(-\frac{\eta^2(x,y)}{ct}\bigg)\,dt,
\end{align*}
which has been already estimated by the right hand side of \eqref{sec:OA,eq:I1}.

Finally, \eqref{sec:OA,eq:smooth} follows after collecting the bounds \eqref{sec:OA,eq:I2} and \eqref{sec:OA,eq:I1}, thus finishing the proof of Lemma \ref{lem:calzyg}.
\end{proof}

Now we focus on the operator $D=T-\tilde{T}.$ Since $T$ and $\tilde{T}$ are associated with the kernels $K$ and $\tilde{K},$ respectively, $D$ is associated with
$$D(x,y)=\int_{0}^1\int_{0}^{1}\kappa(r)\,\partial_r [r^{\mathcal{L}}(x_1,y_1)-e^{\frac{1}{4}(1-r^2)\Delta}(x_1-y_1)]\,r^{A}(x_2,y_2)\,dr.$$
Using \eqref{sec:OA,eq:maxcharrL1}, \eqref{sec:OA,eq:comestt}, and the fact that $\supp \kappa\subseteq [\varepsilon,1-\varepsilon],$  it is not hard to see that
$$\tilde{T}^{glob}f(x)=\int_{\mathbb{R}^d}\int_Y \chi_{N_2^c}\tilde{K}(x,y)f(y)\,d\mu(y_2)\,dy_1,$$
is a well defined and bounded operator on $L^1(\Lambda\otimes \mu).$ Thus,
$$D^{glob}f(x):=T^{glob}f(x)-\tilde{T}^{glob}f(x)=\int_{\mathbb{R}^d}\int_Y \chi_{N_2^c}D(x,y)f(y)\,d\mu(y_2)\,dy_1,$$
is a well defined operator on $L_c^{\infty}.$ Consequently, $D^{loc}f(x):=Df(x)-D^{glob}f(x)$ is also a.e.\ well defined for $f\in L_c^{\infty}.$ Moreover, we have $D^{loc}=T^{loc}-\tilde{T}^{loc},$ where $\tilde{T}^{loc}:=\tilde{T}-\tilde{T}^{glob}.$

We shall need an auxiliary lemma. Recall that $\M_r$ and $\mW_r$ are given by \eqref{sec:OA,eq:MehlformOU} and \eqref{sec:OA,eq:mWrdef}, respectively.
\begin{lem}
\label{lem:diffest}
If $(x_1,y_1)\in N_2,$ then we have
\begin{align}\nonumber
D_I(x_1,y_1)&:=\int_0^1 |\partial_r\mathcal{M}_r(x_1,y_1)-\partial_r\mW_r(x_1-y_1)|\,dr\\
&\leq
\left\{
\begin{array}{rl}
 &C \frac{1+|x_1|}{|x_1-y_1|^{d-1}},  \\
  & C (1+|x_1|)\log \frac{C}{|x_1||x_1-y_1|},
 \end{array} \right.
 {\rm  if  }
\begin{array}{rl}
 & d>1, \\
  & d=1.
 \end{array}
 \label{sec:OA,eq:bounddiffest}
\end{align}
\end{lem}
\begin{proof}
We proceed similarly to the proof of \cite[Lemma 3.9]{high}.
Since for $(x_1,y_1)$ from the local region $N_2$ we have
\begin{equation}
\label{sec:OA,eq:localreg}
|x_1-y_1|^2+(r-1)^2|x_1|^2-C(1-r)\leq |rx_1-y_1|^2\leq (r-1)^2|x_1|^2+|x_1-y_1|^2+C(1-r),
\end{equation}
therefore
$$\frac{|rx_1-y_1|^2}{1-r^2}\exp\bigg(-\frac{|rx_1-y_1|^2}{1-r^2}\bigg)\lesssim \exp\bigg(-\frac{|x_1-y_1|^2}{2(1-r^2)}\bigg),$$
and
\begin{align}
\label{sec:OA,eq:estmix}&|\langle rx_1-y_1,x_1\rangle|\exp\bigg(-\frac{|rx_1-y_1|^2}{2(1-r^2)}\bigg)\\ \nonumber
&\lesssim |rx_1-y_1|\exp\bigg(-\frac{|x_1-y_1|^2}{4(1-r)}\bigg)|x_1|\exp\bigg(-c(1-r)|x_1|^2\bigg)\lesssim 1.
\end{align}
Thus, using \eqref{sec:OA,eq:comM} we obtain for $(x_1,y_1)\in N_2,$
\begin{equation}\label{sec:OA,eq:comMestt} |\partial_r\,\mM_r(x_1,y_1)|\lesssim (1-r)^{-d/2-1}\exp\bigg(-\frac{|x_1-y_1|^2}{4(1-r)}\bigg),\qquad 0<r<1.\end{equation}
Note that the above inequality implies
\begin{equation}\label{sec:OA,eq:comesttintM} \int_0^1|\partial_r\,\mM_r(x_1,y_1)|\lesssim |x_1-y_1|^{-d},\qquad (x_1,y_1)\in N_2.\end{equation}

Using \eqref{sec:OA,eq:comMestt} and \eqref{sec:OA,eq:comestt} we easily see that
$$\int_0^{1/2}|\partial_r\mathcal{M}_r(x_1,y_1)-\partial_r\mW_r(x_1-y_1)|\,dr\lesssim 1,\qquad (x_1,y_1)\in N_2$$
which is even better then the estimate we want to prove.

Now we consider the integral over $(1/2,1).$ Denoting $r(x_1)=\max (1/2, 1-|x_1|^2)$ and using once again \eqref{sec:OA,eq:comMestt} and \eqref{sec:OA,eq:comestt} we obtain
$$\int_{1/2}^{r(x_1)}|\partial_r\mathcal{M}_r(x_1,y_1)-\partial_r\mW_r(x_1-y_1)|\,dr\lesssim \int_{1/2}^{r(x_1)}(1-r)^{-d/2-1}e^{-\frac{|x_1-y_1|^2}{4(1-r)}}\,dr.$$
The above quantity is exactly the one estimated by the right hand side of \eqref{sec:OA,eq:bounddiffest} in the second paragraph of the proof of \cite[Lemma 3.9]{high}. It remains to estimate the integral taken over $(r(x_1),1).$ Using the formulae \eqref{sec:OA,eq:comM} and \eqref{sec:OA,eq:comker} together with \eqref{sec:OA,eq:estmix} we write
\begin{equation*}
\int_{r(x_1)}^1|\partial_r\mM_r(x_1,y_1)-\partial_r\mW_r(x_1-y_1)|\,dr\lesssim J_1+J_2,
\end{equation*}
with
\begin{align*}
J_1=&\int_{r(x_1)}^1 (1-r)^{-d/2-1}\left|\big[d-\frac{2|rx_1-y_1|^2}{1-r^2}\big]\exp\bigg(-\frac{|rx_1-y_1|^2}{(1-r^2)}\bigg)\right.
\\&-\left.\big[d-\frac{2|x_1-y_1|^2}{1-r^2}\big]\exp\bigg(-\frac{|x_1-y_1|^2}{1-r^2}\bigg)\right|\,dr,\\
J_2=&|x_1|\int_{r(x_1)}^1(1-r)^{-d/2-1/2}\exp\bigg(-\frac{|x_1-y_1|^2}{8(1-r)}\bigg)\,dr.
\end{align*}
The quantity $J_2$ has been already estimated in the proof of \cite[Lemma 3.9, p.12]{high}, thus we focus on $J_1.$ For fixed $r, x_1,y_1$ denote
$$\ph(s)=\phi_{r,x_1,y_1}(s)=\left(d-\frac{2|sx_1-y_1|^2}{1-r^2}\right)\exp\bigg(-\frac{|sx_1-y_1|^2}{1-r^2}\bigg),$$
so that
$$J_1=\int_{r(x_1)}^1 (1-r)^{-d/2-1}|\phi_{r,x_1,y_1}(1)-\phi_{r,x_1,y_1}(r)|\,dr.$$

Since
$$\ph'(s)=2\left(-\frac{(d+2)\langle sx_1-y_1,x_1\rangle}{1-r^2}+2\frac{\langle sx_1-y_1,x_1\rangle|sx_1-y_1|^2}{(1-r^2)^2}\right)\exp\bigg(-\frac{|sx_1-y_1|^2}{1-r^2}\bigg),$$
by using \eqref{sec:OA,eq:localreg} and \eqref{sec:OA,eq:estmix} with $r$ replaced by $s,$ we obtain
\begin{equation*}
|\ph'(s)|\lesssim |x_1|(1-r)^{-1/2}\exp\bigg(-\frac{|x_1-y_1|^2}{8(1-r)}\bigg).
\end{equation*}
Thus, by the mean value theorem
\begin{align*}
|J_1|\lesssim |x_1|\int_{r(x_1)}^1(1-r)^{-d/2-1/2}\exp\bigg(-\frac{|x_1-y_1|^2}{8(1-r)}\bigg)\,dr=J_2.
\end{align*}
Recalling that $J_2$ was estimated before, we conclude the proof.
\end{proof}

As a corollary of Lemma \ref{lem:diffest} we now prove the following.
\begin{lem}
\label{prop:dloc}
The operator $D^{loc}$ is bounded on all the spaces $L^p(\Leb\otimes \mu).$  Moreover,
\begin{equation}
\label{sec:OA,eq:DlocLpbound}
\|D^{loc}\|_{L^p(\Leb\otimes \mu) \to L^p(\Leb\otimes \mu)}\lesssim \ki, \qquad 1\leq p\leq \infty.\end{equation}
\end{lem}
\begin{proof} Observe that $D^{loc}$ may be expressed as
$$D^{loc}f(x)=\int_{\mathbb{R}^d}\chi_{N_2}(x_1,y_1)\int_0^1\, \kappa(r) \left[\partial_r\mathcal{M}_r(x_1,y_1)-\partial_r\mW_r(x_1-y_1)\right] (r^A f)(y_1,x_2)\,dr\,dy_1,$$
at least for $f\in L_c^{\infty}.$ Moreover, the estimates below imply that the integral defining $D^{loc}$ is actually absolutely convergent, whenever $f\in L^p(\Lambda\otimes\mu),$ for some $1\leq p\leq \infty.$

Using Fubini's theorem, and the $L^1(Y,\mu)$ contractivity of $r^A,$
\begin{align*}
&\|\kappa\|_{\infty} ^{-1}\|D^{loc}f\|_{L^{1}(\Leb\otimes \mu)}\\
&\leq \int_{\mathbb{R}^d}\int_{\mathbb{R}^d}\chi_{N_2}\int_0^1 |\partial_r\mathcal{M}_r(x_1,y_1)-\partial_r\mW_r(x_1-y_1)|\int_{Y}|(r^A|f|)(y_1,x_2)|\, d\mu(x_2)\,dr\,dy_1dx_1\\
&\leq  \int_{Y}\int_{\mathbb{R}^d}\int_{\mathbb{R}^d}\chi_{N_2}D_I(x_1,y_1)|f(y_1,x_2)| \,dy_1dx_1\,d\mu(x_2).
\end{align*}
Now, using Lemma \ref{lem:diffest} it can be shown that the singularity of $\chi_{N_2}D_I(x_1,y_1)$ is integrable in $x_1.$ Moreover, $\int_{\mathbb{R}^d}\chi_{N_2}D_I(x_1,y_1)\,dx_1\leq C,$ where $C$ is independent of $y_1.$ Thus, applying Fubini's theorem we obtain $\|D^{loc}\|_{L^{1}(\Leb\otimes \mu)\to L^{1}(\Leb\otimes \mu)}\leq C\|\kappa\|_{\infty}$. Since in the local region $ |x_1|\leq 2+ |y_1|\leq 4+|x_2|$ and $\chi_{N_2}(x_1,y_1)=\chi_{N_2}(y_1,x_1),$ the singularity of $\chi_{N_2}D_I(x_1,y_1)$ is also integrable in $y_1.$ Hence, using Fubini's theorem and the $L^{\infty}(Y,\mu)$ contractivity of $r^A,$ we have $\|D^{loc}\|_{L^{\infty}(\Leb\otimes \mu)\to L^{\infty}(\Leb\otimes \mu)}\leq C\|\kappa\|_{\infty}.$ Interpolating between the $L^1(\Leb\otimes \mu)$ and $L^{\infty}(\Leb\otimes \mu)$ bounds for $D^{loc}$ we finish the proof of \eqref{sec:OA,eq:DlocLpbound}.
\end{proof}

The last lemma of this section shows that the local parts of $T$ and $\tilde{T}$ inherit their boundedness properties. Moreover, it says that the operators $T^{loc},$ $\tilde{T}^{loc},$ and $D^{loc}$ are bounded on appropriate spaces with regards to both the measures $\Lambda\otimes \mu$
and $\gamma\otimes \mu.$
\begin{lem}
\label{lem:locinher}
Let $S$ denote one of the operators $T,$ $\tilde{T},$ or $D^{loc}$ and let $\nu$ be any of the measures $\gamma$ or $\Lambda.$ Then $S^{loc}$ is bounded on $L^2(\nu\otimes \mu):=L^2(\mathbb{R}^d\times Y,\nu\otimes \mu),$ and
\begin{equation}
\label{sec:OA,eq:locinherL2}
\|S^{loc}\|_{L^2(\nu\otimes \mu)\to L^2(\nu\otimes \mu)}\lesssim \ki.\end{equation}
Moreover, both $S=\tilde{T}^{loc}$ and $S=D^{loc}$ are of weak type $(1,1)$ with respect to $\nu\otimes \mu,$ with $\nu=\gamma$ or $\nu=\Lambda,$ and
\begin{equation}
\label{sec:OA,eq:locinherweakTtil}
(\nu\otimes\mu)(x\in \mathbb{R}^d\times Y\colon |S^{loc}f(x)|>s)\leq\frac{C_{d,\mu}\ki}{s}\|f\|_{L^1(\nu\otimes \mu)}.
\end{equation}
\end{lem}
\begin{proof}
In what follows $S(x,y)$ denotes the kernel $K(x,y)$ of $T,$ or the kernel $\tilde{K}(x,y)$ of $\tilde{T},$ or the kernel $D^{loc}(x,y)$ of $D^{loc}.$ Recall that in all the cases the integral defining $S^{glob}f(x)$ is absolutely convergent.

The proof is analogous to the proof of \cite[Proposition 3.4]{laptype}. Let $B_j$ be the family of balls in $\mathbb{R}^d$ from Lemma \ref{lem:Lem31laptype}. Take $f\in L_c^{\infty}$ and, for $x_1\in B_j,$ decompose
\begin{align*}
&S^{loc}f(x)=Sf(x)-S^{glob}f(x)\\
&=S(f\chi_{4B_j}(y_1))(x)+S(f\chi_{(4B_j)^c}(y_1))-\int_{\mathbb{R}^d}\int_Y \chi_{N_2^c}S(x,y)(x_1,y_1)f(y)\, d\mu(y_2)\,dy_1\\
&=S(f\chi_{4B_j}(y_1))(x)+\int_{\mathbb{R}^d}\int_Y \chi_{(4B_j)^c}(y_1)-\chi_{N_2^c}(x_1,y_1))S(x,y)f(y)\, d\mu(y_2)\,dy_1.
\end{align*}

Multiplying by $\chi_{B_j}(x_1)$ and summing over $j,$ we arrive at the inequality
\begin{align*}
S^{loc}f(x)&\leq \int_{\mathbb{R}^d}\sum_j \chi_{B_j}(x_1)|\chi_{N_2}(x_1,y_1)-\chi_{4B_j}(y_1)|\int_{Y}|S(x,y)||f(y)|\,d\mu(y_2)\,dy_1 \\
&+\sum_j \chi_{B_j}(x_1)|S(f(y)\chi_{4B_j}(y_1))(x)|
:= S_2(f)+S_1(f),
\end{align*}
Recall that $T$ is bounded on $L^2,$ while $\tilde{T}$ and $D^{loc}$ are bounded on $L^2(\Lambda\otimes \mu).$ Hence, taking $S$ equal to $T,\tilde{T},$ or $D^{loc},$ and using Lemma \ref{lem:Lem33laptype} we see that in all the considered cases $S_1$ is bounded on $L^2(\nu\otimes \mu),$ and $\|S_1\|_{L^2(\nu\otimes \mu)\to L^2(\nu\otimes \mu)}\lesssim \ki.$ Moreover, from Lemmata \ref{lem:calzyg} and \ref{prop:dloc}, we know that both $\tilde{T}$ and $D^{loc}$ are of weak type $(1,1)$ with respect to $\Lambda \otimes \mu,$ and $$\|\tilde{T}\|_{L^1(\Lambda\otimes \mu)\to L^{1,\infty}(\Lambda\otimes \mu)}+\|D^{loc}\|_{L^1(\Lambda\otimes \mu)\to L^{1,\infty}(\Lambda\otimes \mu)}\lesssim \ki.$$ Consequently, using once again Lemma \ref{lem:Lem33laptype}, we see that in both the cases $S=\tilde{T}$ and $S=D^{loc}$ we have $\|S_1\|_{L^1(\nu\otimes \mu)\to L^{1,\infty}(\nu\otimes \mu)}\lesssim \ki.$

It remains to consider $S_2,$ for which we show boundedness on both $L^1(\nu\otimes \mu)$ and $L^{\infty}(\nu\otimes \mu),$ hence, by interpolation on all $L^p(\nu\otimes \mu)$ spaces, $1\leq p\leq \infty.$
 Here we need the following estimates, valid for $f\in L_c^{\infty};$
\begin{equation}
\label{sec:OA,eq:SL1norm}
\int_{Y}\int_Y|S(x,y)f(y_1,y_2)|\,d\mu(y_2)\,d\mu(x_2)\lesssim \ki |x_1-y_1|^{-d}\|f(y_1,\cdot)\|_{L^1(Y,\mu)},
\end{equation}
where $(x_1,y_1)\in N_2,$ and
\begin{equation}
\label{sec:OA,eq:SLinf}
\int_Y|S(x,y)f(y_1,y_2)|\,d\mu(y_2)\lesssim \ki (1+|x_1|)^{d} \|f(y_1,\cdot)\|_{L^{\infty}(Y,\mu)},
\end{equation}
where $(x_1,y_1)\in N_2\setminus N_{1/7}.$ Recall that $r^A$ is a contraction on both $L^1(Y,\mu)$ and $L^{\infty}(Y,\mu).$ Thus, for $S=T$ the bound \eqref{sec:OA,eq:SL1norm} follows from \eqref{sec:OA,eq:comesttintM}, for $S=\tilde{T}$ it is a consequence of \eqref{sec:OA,eq:comesttint}, while for $S=D^{loc}$ it can be deduced from a combination of both \eqref{sec:OA,eq:comesttintM} and \eqref{sec:OA,eq:comesttint}. To prove  \eqref{sec:OA,eq:SLinf} we use the $L^{\infty}$ contractivity of $r^A$ together with the estimates \eqref{sec:OA,eq:comesttintM}, \eqref{sec:OA,eq:comesttint} and fact that $|x_1-y_1|^{-d}\lesssim (1+|x_1|)^{d}$ for $(x_1,y_1)\in N_2\setminus N_{1/7}$.

We start with the boundedness on $L^1(\nu\otimes \mu)$ and denote $g(y_1)=\|f(y_1,\cdot)\|_{L^1(Y,\mu)},$ and $H(x_1,y_1)=\sum_j \chi_{B_j}(x_1)|\chi_{N_2}(x_1,y_1)-\chi_{4B_j}(y_1)|.$ Lemma \ref{lem:Lem31laptype} vii) together with the definition of $B_j$ imply that $H(x_1,y_1)$ is supported in $N_2\setminus N_{1/7}$. Hence, by Fubini's theorem and \eqref{sec:OA,eq:SL1norm},
\begin{align*}\|S_2(f)(x_1,\cdot)\|_{L^1(Y,\mu)}&\lesssim \ki  \int_{\mathbb{R}^d}H(x_1,y_1)S(x,y)\,|g(y_1)|\,d y_1
\\&\lesssim
\ki  \int_{\mathbb{R}^d}H(x_1,y_1)|x_1-y_1|^{-d}\,|g(y_1)|\,d y_1.
\end{align*}
From that point we proceed exactly as in the $T^2$ part of the proof of \cite[Proposition 3.4]{laptype}, arriving at
$
\|S_2(f)\|_{L^1(\nu\otimes \mu)}\lesssim\int_{\mathbb{R}^d}g(y_1)\,d\nu(y_1)= \|f\|_{L^1(\nu\otimes \mu)}.
$

To finish the proof of Lemma \ref{lem:locinher} it remains to show the  $L^{\infty}(\nu\otimes \mu)$ boundedness of $S_2.$ Setting $g(y_1)=\|f(y_1,\cdot)\|_{L^{\infty}(Y,\mu)}$ and using \eqref{sec:OA,eq:SLinf} it follows that
\begin{equation*}
 |S_2(f)(x_1,x_2)|\lesssim \ki (1+|x_1|)^{d}\int_{|x_1-y_1|\leq \frac2{1+|x_1|}}g(y_1)\,dy_1\lesssim \ki \|f\|_{L^{\infty}(\nu\otimes \mu)},
\end{equation*}
as desired.
\end{proof}

Summarizing, since $T^{loc}=\tilde{T}^{loc}+D^{loc},$ from Lemma \ref{lem:locinher} it follows that the local part $T^{loc}$ is of weak type $(1,1)$ with respect to both $\gamma\otimes \mu$ and $\Lambda \otimes \mu.$ Moreover, the weak type $(1,1)$ constant is less than or equal to $C_{d,\mu}\ki.$ Hence, after combining Propositions \ref{prop:proglob} and \ref{thm:tlocalpart}, the proof of Theorem \ref{thm:OA} is completed.

\subsection*{Acknowledgments}
Most of the material of this paper is a part of the PhD thesis of the author \cite{PhD}. The thesis was written under a cotutelle agreement between Scuola Normale Superiore, Pisa, and Uniwersytet Wroc\l awski, and was jointly supervised by Prof.\ Fulvio Ricci and Prof.\ Krzysztof Stempak. I am grateful to both the advisors for all their help and encouragement. Specifically, I thank Prof.\ Ricci, for suggesting the topic of Section \ref{sec:OA}.

The research was partially supported by Polish funds for sciences, NCN Research Project 2011\slash01\slash N\slash ST1\slash01785.
\appendix
\section{Appendix}
\label{App}
As we observed before, besides being bounded from $L^1_{\gamma}(H^1)$ to $L^{1,\infty}(\gamma \otimes \mu)$ and on $L^2,$ the operator $m_{\kappa}(\mL,A)$ is also bounded on all the $L^p$ spaces, $1<p<2.$ In this appendix we show that the interpolation property remains true for general operators.
\begin{thm}
\label{thm:interH1}
Let $S$ be an operator which is bounded from $L^1_{\gamma}(H^1)$ to $L^{1,\infty}(\gamma\otimes \mu),$ and from $L^2$ to $L^2.$ Then $S$ is bounded on all $L^p$ spaces, $1<p<2.$
\end{thm}

The main ingredient of the proof is a Calder\'on-Zygmund decomposition of a function $f(x_1,x_2),$ with respect to the variable $x_2,$ when $x_1$ is fixed, see Lemma \ref{lem:calzygdec}. For the decomposition we present it does not matter that we consider $\mathbb{R}^d$ with the measure $\gamma.$ The important assumption is that $(Y,\zeta,\mu)$ is a space of homogeneous type. Therefore till the end of the proof of Lemma \ref{lem:calzygdec} we consider a more general space $L^1:=L^1(X\times Y, \nu \otimes \mu).$ Here $\nu$ is an arbitrary $\sigma$-finite Borel measure on $X.$ Recall that, by convention, elements of $X$ are denoted by $x_1,y_1,$ while elements of $Y$ are denoted by $x_2,y_2.$

It is known that in every space of homogeneous type in the sense of Coifman-Weiss there exists a family of {\it disjoint} 'dyadic' cubes, see \cite[Theorem 2.2]{dyad}. Here we use \cite[Theorem 2.2]{dyad} to $(Y,\zeta,\mu).$ Let $\mathcal{Q}_l$ be the set of all dyadic cubes of generation $l$ in the space $(Y,\zeta,\mu).$ Note that $l\to\infty$ corresponds to 'small' cubes, while $l\to-\infty$ to 'big' cubes. We define the $l$-th generation dyadic average and the dyadic maximal function with respect to the second variable, by
\begin{equation*}
E_{l}f(x)=\sum_{Q\in\mathcal{Q}_l}\frac{1}{\mu(Q)}\int_{Q}f(x_1,y_2)\,d\mu(y_2)\,\chi_{Q}(x_2),\end{equation*}
and
\begin{equation}\label{sec:OA,eq:dyadmaxdef}\mathcal{D}f(x)=\sup_{l} E_{l}|f|(x),\end{equation}
respectively.

We prove the following Calder\'on-Zygmund type lemma.
\begin{lem}
\label{lem:calzygdec}
Fix $s>0$ and let $f\in L^1$ be a continuous non-negative function on $X\times Y.$  Then there exist Borel measurable functions $g$ ('good') and $\{b_j\}$ ('bad') such that  $f=g+b:=g+\sum_{j} {b_j},$  and:
\begin{enumerate}[(i)]
\item $\|g\|_{L^1}+\sum_j \|b_j\|_{L^1}\leq 4\|f\|_{L^1};$
\item $|g(x)|\leq C_{\mu} s,$ for $x=(x_1,x_2)\in X\times Y;$
\item each function $b_j$ is associated with unique dyadic cube $Q_j.$ Moreover, the functions $b_j$ are supported in disjoint measurable sets $S_j=F_j\times Q_j$ such that for each fixed $x_1\in X,$ we have $\sum_{j}\mu (S_j(x_1))\leq s^{-1}\int_Y f(x)\,d\mu(x_2),$ where $S_j(x_1)=\{x_2\colon x\in S_j\}.$ Additionally, for each fixed $j\in\mathbb{Z}$ and $x_1\in X,$ $\int_{Y}b_j(x)\,d\mu(x_2)=0,$ and either, there exists a 'cube' such that $Q_{j(x_1)}=S_j(x_1)$ and $\supp(b_j(x_1,\cdot))\subset Q_{j(x_1)},$ or $S_j(x_1)=\emptyset$ and $b_j(x_1,\cdot)\equiv 0;$
\item If, for fixed $x_1\in X$ the set $S_j(x_1)$ is non empty (hence in view of (iii) $S_j(x_1)=Q_{j(x_1)}$), then
 $$C_{\mu}^{-1}s\leq \frac{1}{\mu(Q_{j(x_1)})}\int_{Q_{j(x_1)}}f(x)\,d\mu(x_2)\leq C_{\mu} s;$$
\item $$\{x\in X\times Y\colon \mathcal{D}(f)(x)>s\}=\bigcup_{j}F_j\times Q_j=\bigcup S_j.$$
\end{enumerate}
\end{lem}
\begin{proof}
The lemma is intuitively quite clear. The fact we do need to prove is that the decomposition can be done in a 'measurable' way.

Since $f$ is continuous $E_{l}f$ is measurable on $X\times Y.$ Therefore
$$\Omega_l=\left\{x\in X\times Y\colon E_{l}f(x)> s,\quad E_{l'}f(x)\leq s \textrm{ for } l'<l\right\}$$
are measurable subsets of $X\times Y.$ Moreover, the sets $\Omega_l$ are pairwise disjoint and satisfy
\begin{equation}\label{sec:OA,eq:sumOmel} \Omega:=\bigcup_l \Omega_l=\{x\in X\times Y\colon \mathcal{D}(f)(x)>s\}.\end{equation} Setting $\Omega=\bigcup_l\Omega_l$ we see that if $x\in \Omega^c,$ then $f(x)\leq s.$

Observe now that for each fixed $x_1\in X,$ if $z_{Q_{\alpha}^l}$ denotes the center of the cube $Q_{\alpha}^l,$ then $E_{l}f(x)=E_{l}f(x_1,z_{Q_{\alpha}^l}),$ for all $x_2\in Q_{\alpha}^l.$ Therefore, a short reasoning shows that $\Omega_l=\bigcup_{\alpha} F_{\alpha,l} \times Q_{\alpha}^l\equiv \bigcup_{\alpha} S_{\alpha,l},$ where
\begin{equation*}
F_{\alpha,l}=\left\{x_1\in X\colon E_{l}f(x_1,z_{Q_{\alpha}^l})> s,\quad E_{l'}f(x_1,z_{Q})\leq s \textrm{ for } Q\supset Q_{\alpha}^l,\,Q \in \mathcal{Q}_{l'},\,l'<l\right\}.
\end{equation*}
From the continuity of $f$ it follows that the sets $S_{\alpha,l}$ are $\nu \otimes \mu$ measurable. Moreover, $\Omega=\bigcup_{\alpha,l}S_{\alpha,l},$ where the sum runs over $(\alpha,l)$ corresponding to all cubes and the sets $S_{\alpha,l}$ are pairwise disjoint. Hence, recalling \eqref{sec:OA,eq:sumOmel}, we obtain (v).

Note that some of the sets $S_{\alpha,l}$ may be empty, as well as the sets $S_{\alpha,l}(x_1)=\{x_2\colon x\in S_{\alpha,l}\}.$ However, if for some $x_1\in X$ the set $S_{\alpha,l}(x_1)$ is not empty, then $S_{\alpha,l}(x_1)$ coincides with a cube $Q^{\alpha}_l(x_1).$ In fact the just presented construction may be phrased as follows: $x_1\in F_{\alpha,l}$ is and only if the cube $Q^{\alpha}_{l}$ has been chosen as one of the cubes for the Calder\'on-Zygmund decomposition of the function $f(x_1,\cdot).$

Since the set of pairs $(\alpha,l)$ is countable from now on we associate with each $j$ a pair $(\alpha,l)$ and a cube $Q_{\alpha,l}.$ Then $S_j= F_j \times Q_j$ are the sets from (iii). Next we set
\begin{align*}
&g(x)=f(x)\chi_{\Omega^c}+\sum_j\frac{1}{\mu(Q_{j})}\int_{Q_{j}}f(x_1,y_2)\,d\mu(y_2)\,\chi_{S_j}(x),\\
&b(x)=\sum_{j}b_{j}(x)=\sum_{j}\left(f(x)-\frac{1}{\mu(Q_{j})}\int_{Q_{j}}f(x_1,y_2)\,d\mu(y_2)\right)\chi_{S_j}(x),
\end{align*}
so that $f=g+\sum_{j}b_{j}.$ Also, since each set $S_j$ is uniquely associated with the dyadic cube $Q_j,$ the same holds for the functions $b_j.$ Let $x_1\in X$ be fixed. Then either $S_j(x_1)$ is or is not empty. In the second case $S_j(x_1)=Q_{j}(x_1),$ for some cube $Q_{j}(x_1).$ Moreover, the cubes $Q_j(x_1)$ are pairwise disjoint. Hence,
\begin{align*}
&\sum_j\int_{S_j(x_1)}\frac{1}{\mu(Q_{j})}\int_{Q_{j}}f(x_1,y_2)\,d\mu(y_2)\,\chi_{S_j}(x)\,d\mu(x_2)\\
&=\sum_{j\colon S_j(x_1)\neq \emptyset}\int_{Q_j(x_1)}\frac{1}{\mu(Q_{j}(x_1))}\int_{Q_{j}(x_1)}f(x_1,y_2)\,d\mu(y_2)\,\,d\mu(x_2)\leq \int_Y f(x)\,d\mu(x_2)
\end{align*}
and consequently, since $\chi_{S_j}(x)=\chi_{S_j(x_1)}(x_2),$ using Fubini's theorem we obtain
\begin{align*}
&\iint\limits_{X\times Y}\sum_j\frac{1}{\mu(Q_{j})}\int_{Q_{j}}f(x_1,y_2)\,d\mu(y_2)\,\chi_{S_j}(x)\, d\nu(x_1)\,d\mu(x_2)\\
&=\int_X \bigg(\sum_j\int_{S_j(x_1)}\frac{1}{\mu(Q_{j})}\int_{Q_{j}}f(x_1,y_2)\,d\mu(y_2)\,\chi_{S_j}(x)\,d\mu(x_2)\bigg)\,d\nu(x_1)\\
&\leq \int_{X\times Y}f(x_1,y_2)\,d\mu(y_2)\,d\nu(x_1)=\|f\|_{L^1}. 
\end{align*}
From the above we obtain $\|g\|_{L^1}\leq 2\|f\|_{L^1}$ and $\sum_j \|b_j\|_{L^1}\leq  2\|f\|_{L^1},$ thus proving (i).

Now we pass to (ii). Since $|f(x)|\leq s,$ for $x\in \Omega^c$ and the sets $S_j$ are disjoint it suffices to show that,
\begin{equation}
\label{sec:OA,eq:toshow}
\frac{1}{\mu(Q_{j})}\int_{Q_{j}}f(x_1,y_2)\,d\mu(y_2)\chi_{S_j}(x)\leq Cs,\qquad \textrm{ for } x_2\in S_{j}(x_1).
\end{equation} If $x_2\in S_j(x_1),$ then $S_{j}(x_1)=Q_{j}(x_1),$ for some $Q_{j}(x_1)\in\mathcal{Q}_l.$ Moreover, there exists $\tilde{Q}_{j}(x_1)\supset Q_{j}(x_1),$ with $\tilde{Q}_{j}(x_1)\in \mathcal{Q}_{l-1}.$ Then, since $x_2\in S_j(x_1),$  $$\frac{1}{\mu(\tilde{Q}_{j}(x_1))}\int_{\tilde{Q}_{j}(x_1)}f(x_1,y_2)\,d\mu(y_2)=E_{l'}f(x)\leq s.$$ Therefore, a standard argument, based on the doubling property of $\mu,$ gives
$$\frac{1}{\mu(\tilde{Q}_{j}(x_1))}\int_{\tilde{Q}_{j}(x_1)}f(x)\,d\mu(y_2)\leq \frac{C_{\mu}}{\mu(\tilde{Q}_{j}(x_1))}\int_{\tilde{Q}_{j}(x_1)}f(x_1,y_2)\,d\mu(y_2)\leq C_{\mu} s.$$ Hence, \eqref{sec:OA,eq:toshow} and thus also (ii) is proved.

Observe that from the very definition of the sets $S_j$ we have
\begin{equation}
\label{sec:OA,eq:comless}
\frac{1}{\mu(Q_{j})}\int_{Q_{j}}f(x_1,y_2)\,d\mu(y_2)\chi_{S_j}(x)> s,\qquad \textrm{ for } x_2\in S_{j}(x_1).
\end{equation}
Combining the above with \eqref{sec:OA,eq:toshow} we obtain item (iv).

It remains to prove the property (iii). The inequality $\sum_{j}\mu(S_j(x_1))\leq s^{-1}\int_{Y}f(x)\,d\mu(x_2)$ follows from \eqref{sec:OA,eq:comless}. If $S_j(x_1)=\emptyset$ then obviously, $b_j(x_1,\cdot)=0.$ If $S_j(x_1)$ is not empty, then $S_j(x_1)=Q_{j}(x_1),$ for some $j(x_1),$ so that $\supp b_j(x_1,\cdot)\subset Q_{j}(x_1).$ In either case $\int_{Y}b_j(x)\,d\mu(x_2)=\int_{S_j(x_1)}b_j(x)\,d\mu(x_2)=0.$
\end{proof}
Using Lemma \ref{lem:calzygdec} we now prove Theorem \ref{thm:interH1}. The proof follows the scheme from \cite[Theorem D, pp.\ 596, 635--637]{CW} by Coifman and Weiss.
\begin{proof}[Proof of Theorem \ref{thm:interH1}]
Fix $1<q<p$ and set $\mD^q(f)=(\mD(|f|^q))^{1/q},$ with $\mD$ given by \eqref{sec:OA,eq:dyadmaxdef}. Then, since $\mD$ is bounded on $L^p$ and $1<q<p,$ the same is true for $\mD^q.$

Fix a continuous function $0\leq f\in L^p$ and let
\begin{equation}\label{sec:OA,eq:levelset}\Theta^{s}=\{x\colon \mD^{q}(f)>s\}=\{x\colon \mD(f^q)>s^q\}.\end{equation} From item (v) of Lemma \ref{lem:calzygdec} it follows that $$\Theta^{s}=\bigcup_j F_j \times Q_j=\bigcup S_j,$$ where the sets $S_j$ satisfy properties (i)-(iv) from Lemma \ref{lem:calzygdec} with $s^{q}$ in place of $s$ and $f^q$ in place of $f.$ In particular
\begin{equation}
\label{sec:OA,eq:meanprop}
\frac{1}{\mu(Q_j)}\int_{Q_j}f^q\,d\mu(x_2)\approx s^{q},\qquad x_1\in F_j.
\end{equation}

Decompose $f=g_{s}+b_{s}=g_{s}+\sum_j b_{j,s}$ with
\begin{align*}
g_{s}&=g=f(1-\chi_{\Theta^{s}})+\sum_j\frac{1}{\mu(Q_j)}\int_{Q_j}f(x)\,d\mu(x_2)\chi_{S_j}\\
b_{j,s}&=b_j=\left(f(x)-\frac{1}{\mu(Q_j)}\int_{Q_j}f(x_1,y_2)\,d\mu(y_2)\right)\chi_{S_j}.
\end{align*}
If we fix $x_1\in F_j,$ then because $|b_j|\leq |f|+\frac{1}{\mu(Q_j)}\int_{Q_j}f(x_1,y_2)\,d\mu(y_2)\chi_{S_j},$ using \eqref{sec:OA,eq:meanprop} and H\"older's inequality, we obtain
\begin{equation}
\label{sec:OA,eq:dyad-ball}
\begin{split}&\left(\int_{Q_j}|b_j|^{q}\,d\mu(x_2)\right)^{1/q}\\&\leq \left(\int_{Q_j}|f|^{q}\,d\mu(x_2)\right)^{1/q}+\left(\int_{Q_j}\left|\frac{1}{\mu(Q_j)}\int_{Q_j}f(x_1,y_2)\,d\mu(y_2)\right|^{q}\chi_{Q_j}\,d\mu(x_2)\right)^{1/q}\\
&\lesssim s\mu(Q_j)^{1/q}+\left(\int_{Q_j}\frac{1}{\mu(Q_j)}\int_{Q_j}f^q(x_1,y_2)\,d\mu(y_2)\chi_{Q_j}\,d\mu(x_2)\right)^{1/q}\lesssim s \mu(Q_j)^{1/q}.
\end{split}
\end{equation}
Let $\underline{B}(Q_j)$ be the ball included in $Q_j$ from \cite[Theorem 2.2 (2.8)]{dyad}, i.e.\ satisfying
$$\underline{B}(Q_j)\subset Q_j ,\qquad \mu(Q_j)\leq C_{\mu}\,\mu (\underline{B}(Q_j)).$$
Then, from \eqref{sec:OA,eq:dyad-ball} it follows that
$$\left(\frac{1}{\mu(\underline{B}(Q_j))}\int_{\underline{B}(Q_j)}|b_j|^{q}\,d\mu(x_2)\right)^{1/q}\leq C_{\mu} s.$$
Consequently, for each fixed $x_1\in F_j,$ the function $$c_j(x_1,\cdot)=\frac{b_j(x_1,\cdot)}{C_{\mu}s \mu(\underline{B}(Q_j))}$$ is supported in $\underline{B}(Q_j)$ and satisfies \begin{equation}\label{sec:OA,eq:atom}\|c_j(x_1,\cdot)\|_{L^q(Y,\frac{1}{\mu(\underline{B}(Q_j))}d\mu)}\leq \frac{1}{\mu(\underline{B}(Q_j))}.\end{equation}
The above inequality is also trivially satisfied if $x_1\not\in F_j,$ since then $c_j(x_1,\cdot)\equiv 0.$

From \eqref{sec:OA,eq:atom} it follows that for each fixed $x_1\in \mathbb{R}^d$ the functions $c_j(x_1,\cdot)$ are $H^{1,q}(Y,\mu)$-atoms in the sense of Coifman-Weiss \cite[p.\ 591]{CW}, and thus $\|c_j\|_{H^{1,q}(Y,\mu)}=1$. Moreover, from the decomposition $b=\sum_j C_{\mu}s \mu(\underline{B}(Q_j))c_j$ we obtain
$$\|b(x_1,\cdot)\|_{H^{1,q}(Y,\mu)}\leq C_{\mu} s \sum_{j\colon x_1\in F_j}\mu(Q_j)=c_{\mu}s \sum_{j}\mu(S_j(x_1)).$$
Since the spaces $H^{1,q}(Y,\mu)$ and $H^{1}(Y,\mu)=H^{1,\infty}(Y,\mu)$ coincide, cf. \cite[Theorem A]{CW}, using Fubini's theorem and the disjointness of $S_j$ we obtain
\begin{equation}
\label{sec:OA,eq:H1part}
\|b\|_{L^1_{\gamma}(H^1)}=\int_{\mathbb{R}^d}\|b(x_1,\cdot)\|_{H^{1}(Y,\mu)}\,d\gamma(x_1)\lesssim s \sum_{j}(\gamma\otimes \mu)(S_j)=C s (\gamma\otimes \mu)(\Theta^{s}).
\end{equation}

By the layer-cake formula we have
\begin{equation*}
p^{-1}\|Sf\|_{L^p}=\int_0^{\infty}s^{p-1}(\gamma\otimes \mu)(x\colon |Sf(x)|>s)\,ds,
\end{equation*}
and, consequently,
\begin{align*} &\|Sf\|_{L^p}\lesssim\int_0^{\infty}s^{p-1}(\gamma\otimes \mu)(x\colon |Sb_{s}(x)|>s/2)\,ds+
\int_0^{\infty}s^{p-1}(\gamma\otimes \mu)(x\colon |Sg_{s}(x)|>s/2)\,ds\\
&:= E_1+E_2.
\end{align*}
To estimate $E_1$ we use the weak type property of $S$ and \eqref{sec:OA,eq:H1part}, obtaining
\begin{equation}
\label{sec:OA,eq:thtaest}
E_1\lesssim\int_0^{\infty}s^{p-2}\|b_{s}\|_{L^1_{\gamma}(H^1)}\,ds\lesssim\int_0^{\infty}s^{p-1} (\gamma\otimes \mu)(\Theta^{s})\,ds=\|\mD^q(f)\|_p^p\lesssim\|f\|_{p}^p.
\end{equation}

Passing to $E_2,$ the layer-cake formula together with the $L^2$ boundedness of $S$ and Chebyshev's inequality produce
\begin{align*}
&p^{-1}E_2\lesssim C \int_0^{\infty}s^{p-3}\|g_{s}\|_2^2\,ds\\
&=\int_0^{\infty}s^{p-3}\int_{\Theta^{s}}|g_{s}|^2\,d\gamma\,d\mu\,ds+\int_0^{\infty}s^{p-3}\int_{(\Theta^{s})^c}|g_{s}|^2\,d\gamma\,d\mu\,ds:= E_{2,1}+E_{2,2}.
\end{align*}
From \eqref{sec:OA,eq:levelset}, \eqref{sec:OA,eq:meanprop} and the definition of $g_{s}$ we see that
$|g_{s}|\leq C s,$ and consequently,
\begin{align*}
E_{2,1}\leq \lesssim\int_0^{\infty}s^{p-1}(\gamma\otimes \mu)(\Theta^{s})\,ds.
\end{align*}
The above quantity has already been estimated, see \eqref{sec:OA,eq:thtaest}. Now we focus on $E_{2,2}.$ Since $g_{s}=f$ outside of $\Theta^{s}$ and $f\leq\mD^q(f),$ using Fubini's theorem we have
\begin{align*}
E_{2,2}\lesssim\int_{\mathbb{R}^d\times Y}|f(x)|^2\int_{f}^{\infty}s^{p-3}\,ds \,d(\gamma\otimes \mu)\lesssim\int_{\mathbb{R}^d\times Y}|f(x)|^p\,d(\gamma\otimes \mu),
\end{align*}
thus obtaining the desired estimate for $E_2$ and hence, finishing the proof of Theorem \ref{thm:interH1}.
\end{proof}

\end{document}